\documentclass[11pt]{article}

\usepackage[english]{babel}
\usepackage[utf8]{inputenc}
\usepackage[T1]{fontenc}
\usepackage[toc,page]{appendix}
\usepackage{graphicx}
\usepackage{geometry}
\usepackage{bbm}
\usepackage{dsfont}
\usepackage{alltt}
\usepackage{verbatimbox}
\usepackage{amsthm,amsfonts}
\usepackage{amsmath}
\usepackage{amssymb}
\usepackage{mathabx}
\usepackage{accents}
\usepackage{titlesec}
\usepackage{mathtools}
\usepackage{empheq}
\usepackage{subcaption}
\mathtoolsset{showonlyrefs=true}

\newtheorem{thm}{Theorem}[section]
\newtheorem{lem}[thm]{Lemma}

\newtheorem{cor}[thm]{Corollary}

\theoremstyle{definition}
\newtheorem{defi}[thm]{Definition}
\newtheorem{rem}[thm]{Remark}

\newcommand\R{{\mathbb R}}
\newcommand\C{{\mathbb C}}

\newcommand\N{{\mathbb N}}

\newcommand\A{\boldsymbol{\alpha}}

\newcommand\MScN[1]{\href{http://www.ams.org/mathscinet-getitem?mr=#1}{\nolinkurl{(#1)}}}
\newcommand\DOI[1]{\href{http://dx.doi.org/#1}{(doi: \nolinkurl{#1})}}
\newcommand\LINK[1]{\href{#1}{(link: \nolinkurl{#1})}}

\title{A note on bifurcations from eigenvalues of the Dirichlet-Laplacian with arbitrary multiplicity}
\author{Sim\~ao Correia* and M\'ario Figueira}

\begin{document}

\maketitle

 \begin{abstract}
 	In this short note, we consider the elliptic problem
 	$$
 	\lambda \phi + \Delta \phi = \eta|\phi|^\sigma \phi,\quad \phi\big|_{\partial \Omega}=0,\quad \lambda, \eta \in \C,
 	$$
 	on a smooth domain $\Omega\subset \R^N$, $N\ge 1$. The presence of complex coefficients, motivated by the study of complex Ginzburg-Landau equations, breaks down the variational structure of the equation. We study the existence of nontrivial solutions as bifurcations from the trivial solution. More precisely, we characterize the bifurcation branches starting from eigenvalues of the Dirichlet-Laplacian of arbitrary multiplicity. This allows us to discuss the nature of such bifurcations in some specific cases. We conclude with the stability analysis of these branches under the complex Ginzburg-Landau flow.
 \end{abstract}
  
  \section{Introduction}
  
  \subsection{Description of the problem and main results}
  Consider the complex Ginzburg-Landau equation on a smooth domain $\Omega\subset \R^n$ and Dirichlet boundary conditions:
  \begin{equation}\tag{CGL}\label{CGL} 
  \begin{cases}
  v_t = e^{i \theta} \Delta v + e^{i \gamma} |v|^\sigma v + k v, \quad v = v(x, t),\, (x,t)\in \Omega\times[0,T]\\
  v\big|_{\partial \Omega\times(0,T)}=0
  \end{cases} \mbox{for }\gamma, k\in \R,\ |\theta|<\pi/2,\ \sigma>0,
  \end{equation}
The Ginzburg-Landau
equation is a model for several physical and chemical phenomena
such as superconductivity or chemical turbulence. We refer \cite{DGHN}, \cite{Mielke} and the references
cited therein. A very interesting remark is the fact that \eqref{CGL} can be viewed as
a dissipative version of the nonlinear Sch\"odinger equation, which admits
solutions developing localized singularities in finite time. Local existence, 
global existence and uniqueness of solutions of \eqref{CGL} are widely studied on both $\R^N$
or a domain $\Omega \subset \R^N$ under various boundary conditions and
assumptions on the parameters; see \cite{DGL,GV1,GV2,Okaza} and 
the references therein. On the other hand, there are not many results concerning the blow-up of the solutions
of \eqref{CGL}: we refer e.g. \cite{CDW-Blowup} and \cite{Masmou}.

In the analysis of evolution partial differential equations which possess a gauge invariance (see Definition \ref{defi:gauge}),  one may look for particular solutions - \textit{bound-states} -  of the form $e^{i\omega t}u(x)$. In the case of the the Ginzburg-Landau equation above,
the profile $u$ must satisfy the stationary problem
\begin{equation}\label{eq:SW}
	i \omega u = e^{i \theta} \Delta u + e^{i \gamma} |u|^\sigma u + k \phi\quad \mbox{ in } \Omega, \quad u\big|_{\partial \Omega}=0.
\end{equation}
We rewrite equation \eqref{eq:SW} in the more convenient way
  \begin{equation}\tag{BS}\label{BS}
	\lambda u + \Delta u = \eta|u|^\sigma u,\quad \lambda, \eta \in \C.
	%,\quad \lambda:=(k-i\omega)e^{-i\theta}, \ f(x,\phi) :=-e^{i\gamma-\theta}|\phi|^\sigma \phi.
\end{equation}
For $\lambda,\eta\in \R$, there is an extensive theory on the existence and qualitative properties of solutions to \eqref{BS}. On one hand, one may take advantage of the variational structure (that is, the fact that the equation corresponds to the critical points of a well-behaved functional) to construct solutions via minmax arguments, depending on the signs of $\lambda$, $\eta$ (for $\eta<0$, one usually requires a subcritical nonlinearity, $\sigma<4/(N-2)^+$). The literature on variational arguments is too vast, we simply refer to the books \cite{AmbrosettiArcoya}, \cite{AmbrosettiMalchiodi} and references therein. On the other hand, one may explore the existence of small solutions through the application of bifurcation arguments, starting from the eigenvalue problem $\lambda\phi+\Delta\phi=0$. For simple eigenvalues, a direct Implicit Function Theorem argument can be applied. For eigenvalues of odd multiplicity, the existence of bifurcation branches can be shown using the Leray-Schauder topological degree (\cite{Krasn}). For general multiple eigenvalues, one may apply a Lyapunov-Schmidt reduction and solve a nonlinear problem on the corresponding eigenspace. This approach is quite standard, see for example \cite{AmbrosettiProdi,CrandallRabinowitz,Dancer,DelPinoGarciaMusso,MugnaiPistoia}. In the presence of symmetries (e.g. when $\Omega$ is a disk), more refined results can be found in \cite{Miyamoto} and \cite{Dancer2}. Finally, in some very specific cases, one can characterize exactly all bifurcation branches: see \cite{DelPinoGarciaMusso}, for the second eigenvalue in the square,  and also \cite{MugnaiPistoia}, for either the case of general eigenvalues in a rectangle or the second eigenvalue in a cube.

For complex $\lambda$ and $\eta$, the known results are much more scarce. In fact, the variational structure collapses and the min-max arguments available in the real-coefficients case are of no use. Another approach is to reduce the general \eqref{BS} (via a nontrivial transformation) to the real coefficients case. This approach, available for $\Omega=\R$ (\cite{MS_GL1}), seems innaplicable for any other domain. This leaves us reduced to the construction of solutions via bifurcation arguments. In the past decade, there have been some works in this direction \cite{CDW, CDP,MS_GL1,MS_GL2}, where one applies a bifurcation argument starting from a solution to \eqref{BS} with real coefficients. Observe that, once one considers complex coefficients, all eigenvalues have \textit{even} real multiplicity and one cannot apply the results of \cite{Krasn}. The results of \cite{Dancer} also do not fit our framework, as the general bifurcation results presented therein are restricted to the real-valued case.

%
%Focusing on the bifurcation approach, one strategy is to start from the linear equation $\lambda\phi + \Delta \phi=0$ and add the nonlinear term as a small perturbation. The case of simple eigenvalues was handled in \cite{CDW}, while the case of double eigenvalues was studied in \cite{MS_GL1}. For simple eigenvalues, a direct Implicit Function Theorem argument can be applied. For multiple eigenvalues, one must first apply a Lyapunov-Schmidt reduction and solve a nonlinear problem on the corresponding eigenspace. This approach is quite standard, see for example \cite{AmbrosettiProdi,CrandallRabinowitz,Dancer,DelPinoGarciaMusso,MugnaiPistoia}. As far as we know, it seems that the problem of determining the exact number and nature of bifurcation branches is not clearly studied in literature. For example, in \cite{DelPinoGarciaMusso}, the authors focus on the bifurcations starting from the second eigenvalue in the square and characterize all the branches. On the other hand, the case of general eigenvalues in a rectangle (or the second eigenvalue in a cube) was studied in \cite{MugnaiPistoia}. These existence results can be also framed in the general theory of Dancer \cite{Dancer}; however, to understand the nature of such bifurcations, one must rely on the specificities of the problem at hand.

Our main goal in this short note is to state precisely a bifurcation result for \eqref{BS} in the spirit of \cite{CrandallRabinowitz}, starting from eigenvalues of any multiplicity and for $\eta\in \C$. The first result gives sufficient conditions for the existence of bifurcation branches (cf. Definition \ref{defi:bifurc} below):

\begin{thm}\label{thm:exist}  Fix $N\ge 1$. Let $\Omega\subset \R^N$ be a bounded domain and $1 \leq \sigma \leq 4/(N-2)^+$. Let 
	$\tilde{\lambda}$ be an eigenvalue of the Dirichlet-Laplace operator with multiplicity $p\geq 2$, with $L^2$-orthogonal
	real-valued eigenfunctions $u_1, \dots, u_p$. Suppose that the system
	\begin{equation}\label{polinomios}
		\begin{split}
			&P_m(\alpha_2, \dots, \alpha_p) = \int_\Omega
			\left|u_1+\sum_{ j=2}^p \alpha_j u_j\right|^\sigma \left(u_1+\sum_{ j=2}^p \alpha_j u_j\right) 
			(\alpha_m u_1 - u_m)dx = 0,\quad 
			m=2, \dots, p
		\end{split}
	\end{equation}
	has a solution ${\A}^0 = (\alpha_2^0,\ldots, \alpha_p^0)\in \C^{p-1}\simeq \R^{2p-2}$ such that
	\begin{equation}\label{eq:jacobiano}
		\frac{\partial(P_2,\ldots, P_p)}{\partial(\alpha_2,\ldots,\alpha_p)}(\A^0) \neq 0.
	\end{equation}
	Then there exists $\delta > 0$ 
	and a Lipschitz mapping
	$$ \varepsilon\in [0, \delta) \rightarrow (y, \lambda , \A)\in H_0^1(\Omega)\times \C \times \C^{p-1} $$
	with $(y(0), \lambda(0), \A(0)) = (0, \tilde{\lambda}, \A^0)$ and 
	$(y,u_i)_{L^2(\Omega)} = 0, i=1, \ldots, p$, such that
	\begin{equation}\label{estim-sw}
		(\lambda(\epsilon),u(\epsilon)) :=\left(\lambda(\epsilon), y(\varepsilon) + \varepsilon u_1 + \varepsilon \sum_{ j=2}^p \alpha_j(\varepsilon) u_j\right)  
	\end{equation}
	is a bifurcation branch  starting from $(\tilde{\lambda},0)$. Moreover,
	$ \| y(\varepsilon) \|_{H^1_0} \leq C \varepsilon^{\sigma+1} $ and
	\begin{equation}\label{lambda}
		\lambda(\varepsilon) =  \tilde{\lambda} +\eta \varepsilon^{\sigma}
		\int_\Omega \left|u_1+\sum_{ j=2}^p \alpha_j^0 u_j\right|^\sigma \left(u_1+\sum_{ j=2}^p \alpha_j^0 u_j\right)   u_1  dx + o(\varepsilon^\sigma)
	\end{equation}
\end{thm}

\begin{rem}
	We observe that the conditions presented in the above theorem are  completely independent on the parameter $\eta\in \C$. This is not to say that the branch is independent on $\eta$ (when $\eta$ is not real, both $\lambda$ and $u$ cannot be real, by direct integration of \eqref{BS}).
\end{rem}

\begin{rem}
We only guarantee that the resulting bifurcation branches are Lipschitz continuous because we will apply an Implicit Function Theorem for Lipschitz functions. This could be further improved, but we do not pursue this here. In \cite{DelPinoGarciaMusso} and \cite{MugnaiPistoia}, the resulting branches are $C^1$.
\end{rem}

\begin{rem}
	If $\eta\in \R$ and the assumptions of Theorem \ref{thm:exist} are satisfied for some $\A_0\in \R^{p-1}$, then the corresponding bifurcation branch will be real-valued. Indeed, it suffices to go through the proof of Theorem \ref{thm:exist} and replace the field $\C$ with $\R$. 
\end{rem}

 \begin{rem}\label{rem:simples}
	Our methodology could be applied to recover the bifurcation result for simple eigenvalues from \cite{MS_GL1} (analogous to the real-valued case \cite{AmbrosettiProdi,CDW,CrandallRabinowitz}): if $\tilde{\lambda}$ is a simple eigenvalue of the Dirichlet-Laplace operator with eigenfunction $u_1$,
	there exists a unique bifurcation branch (modulo gauge symmetry) in the neighbourhood of $(\tilde{\lambda},0)$ defined by
	$$\phi = y(\varepsilon) + \varepsilon u_1 \quad {\it and}  \quad \lambda(\varepsilon) =  \tilde{\lambda} +\eta \varepsilon^{\sigma} \int_\Omega |u_1|^\sigma u_1^2  dx + o(\varepsilon^\sigma).$$
\end{rem}

Our second result complements Theorem \ref{thm:exist}, as it relates bifurcation branches to solutions of \eqref{polinomios}.

\begin{thm}\label{thm:caract}
	 Fix $N\ge 1$. Let $\Omega\subset \R^N$ be a bounded domain and $1 \leq \sigma < 4/(N-2)^+$. Let 
	 $\tilde{\lambda}$ be an eigenvalue of the Dirichlet-Laplace operator with multiplicity $p\geq 2$, with $L^2$-orthogonal
	 real-valued eigenfunctions $u_1, \dots, u_p$. Suppose that $(\lambda,u):[0,\delta)\to \C\times H^1_0(\Omega)$ is a bifurcation branch starting at $(\tilde{\lambda},0)$. Given $\epsilon_n\to 0$,  there exists $u_0\in H^1_0(\Omega)\setminus\{0\}$ with $-\Delta u_0=\tilde{\lambda}u_0$ such that, up to a subsequence, $u(\epsilon_n)/\|u(\epsilon_n)\|_{L^\infty}\to u_0$ in $H^1_0(\Omega)\cap L^\infty(\Omega)$. After a possible reordering of $u_1,\dots,u_p$, write
	$$
	u_0=c(u_1+\alpha_2u_2+\dots + \alpha_p u_p), \quad c,\alpha_2,\dots,\alpha_p\in \C,\ c\neq 0.
	$$
	Then $\A=(\alpha_2,\dots,\alpha_p)$ is a solution to \eqref{polinomios}.
	
	%  	$2.$ If $\lambda_1$ is a simple eigenvalue of the Dirichlet-Laplace operator with the eigenfunction $u_1$,
	%  	we have now a unique bifurcation branch in the neighbourhood of $(0, \lambda_1)$ defined by
	%  	$$\phi = y(\varepsilon) + \varepsilon u_1 \quad {\it and}  \quad \lambda(\varepsilon) = e^{i\theta} \lambda_1 - e^{i \gamma} \varepsilon^{\sigma} \int_\Omega |u_1|^\sigma u_1^2  dx + o(\varepsilon^\sigma)$$
	%  	
\end{thm}

As a consequence, if all solutions to \eqref{polinomios} satisfy the nondegeneracy condition \eqref{eq:jacobiano}, one may characterize all bifurcation branches:

\begin{cor}\label{cor:equiv}
	Under the conditions of Theorem \ref{thm:caract}, suppose that all solutions to \eqref{polinomios} satisfy the nondegeneracy condition \eqref{eq:jacobiano}. Then, for all $\epsilon\in(0,\delta)$ sufficiently small, there exists a bifurcation branch $(\lambda',u'):[0,\delta')\to \C\times H^1_0(\Omega)$, built through the application of Theorem \ref{thm:exist}, and $\epsilon'\in (0,\delta')$ such that $(\lambda(\epsilon),u(\epsilon))=(\lambda'(\epsilon'),u'(\epsilon'))$, modulo gauge invariance (cf. Definition \ref{defi:gauge}).
\end{cor}

\begin{rem}
	Our results can also be extended (in the same spirit as \cite{MS_GL2,DelPinoGarciaMusso}) to the more general case $\lambda  u  + \Delta  u  =\eta|u|^\sigma u(1+r(x,|u|))$, with $r\in C^1(\Omega\times \R^+, \C)$ such that
	\begin{itemize}
		\item $r(x,0)=0$ for all $x\in \Omega$;
		\item There exists $\sigma_1>0$ with $\sigma+\sigma_1<4/(N-2)^+$ such that $|r(x,y)|\lesssim 1+|y|^{\sigma_1}$ for all $(x,y)\in \Omega\times \R^+$.
	\end{itemize}
\end{rem}

For specific domains, such as the $N$-dimensional rectangular parallelepiped, the conditions of Corollary \ref{cor:equiv} are verified, recovering in particular the results of \cite{DelPinoGarciaMusso} and \cite{MugnaiPistoia}. Notice, however, that in \cite{DelPinoGarciaMusso, MugnaiPistoia}, the authors obtain $C^1$ branches (and not just Lipschitz) and also compute the Morse index of each bifurcation branch.

 \begin{cor}\label{cor:cubo}
	Take $\Omega$ be an $N$-dimensional rectangular parallelepiped and $\sigma=2$ and fix an eigenvalue $\tilde{\lambda}$ with multiplicity $p$. If $p\ge 4$, suppose additionally that any linearly independent eigenfunctions associated to $\tilde{\lambda}$, $u_1,\dots, u_4$, satisfy
\begin{equation}\label{eq:hipotese4}
	\int_\Omega u_1u_2u_3u_4 =0.
\end{equation}
	  Then the real solutions of \eqref{polinomios} are given by $(\alpha_2,\dots \alpha_p)\in \{-1,0,1\}^{p-1}$ and there exist exactly $\frac{3^p-1}{2}$ real bifurcation branches starting from $(\tilde{\lambda},0)$ (modulo gauge invariance).
\end{cor}
\begin{rem}
	In \cite{MugnaiPistoia}, assumption \eqref{eq:hipotese4} is missing. If one considers the function $J_\lambda$ defined therein, one finds some terms  involving the product of four different eigenfunctions. However, in their proof, these terms were overlooked. While \eqref{eq:hipotese4} could be true for $N\ge 2$ independently on $\tilde{\lambda}$, a proof of this statement is unavailable.
\end{rem}

\begin{rem}
	Assumption \eqref{eq:hipotese4} does not hold in general. Indeed, if $\Omega=(0,\pi)^4$, one may check that
	$$
	u_m(x_1,x_2,x_3,x_4)=\prod_{j=1}^4\sin(k_j^mx_j),\quad (k_1^m,\dots,k_4^m)=\begin{cases}
		(1,2,3,4),\ m=1\\
		(2,3,4,1),\ m=2\\
		(3,4,1,2),\ m=3\\
		(4,1,2,3),\ m=4,
	\end{cases}
	$$
	do not satisfy \eqref{eq:hipotese4}, despite being associated to the same eigenvalue.
\end{rem}

\begin{rem}
In the case of the square $\Omega=(0,\pi)^2$, the second eigenvalue is double, with eigenfunctions
$$
u_1(x,y)=\sin(x)\sin(2y),\quad u_2(x,y)=\sin(2x)\sin(y).
$$
A simple computation (see the proof of Corollary \ref{cor:cubo}) shows that
$$
P_2(\alpha_2)=\frac{\pi^2}{8}  (\alpha^2-1)\Re\alpha_2 -\frac{5\pi^2}{64} (|\alpha_2|^2-1)\alpha_2.
$$
Apart from the real solutions $\alpha_2=0,\pm 1$, one can also check that $\alpha_2=\pm i$ is also a solution. Therefore, the corresponding bifurcation branch will be strictly complex, \textit{even if $\eta$ is real}.
\end{rem}

\begin{rem}
	For double eigenvalues, system \eqref{polinomios} reduces to a single equation. If $\sigma$ is even and we restrict ourselves to the real case $\alpha_2\in \R$, $P_2$ becomes a polynomial on $\alpha_2$. Assumming that $P_2\nequiv 0$, the maximum number of bifurcation branches of the form \eqref{estim-sw} is $\sigma+2$.
\end{rem}

\begin{rem}
	In the case of the disk  $\Omega=\{(x,y)\in\R^2 : x^2+y^2=1\}$, it is known (see \cite{Damascelli} and \cite[Section 2.2]{Miyamoto}) that the second eigenvalue $\lambda_2$ is double. In polar coordinates, the eigenspace is generated by $u_1=J(r)\cos(\theta)$ and $u_2=J(r)\sin(\theta)$, where $J$ is the Bessel function of first kind of order 1. Then
	$$
P_2(\alpha_2)= \int_\Omega |u_1+\alpha_2u_2|^2(u_1+\alpha_2u_2)(\alpha_2u_1-u_2) dx =C\left((\alpha_2^2-|\alpha_2|^2)\alpha_2 + 2i\Im \alpha_2\right),\quad C\neq 0.
	$$
	If $\alpha_2\in \R$, $P_2\equiv 0$ and, 
	in this case, our result is not applicable. This is to be expected, since it was proven in \cite{Miyamoto} that there exists a continuum of real bifurcation branches stemming from $\lambda_2$. For general real bifurcation results under the presence of symmetry, the reader may also consult \cite{Dancer2}.
%	However, the problem admits the solutions $\alpha_2=\pm i$, which are nondegenerate. Thus the construction of strictly complex bifurcation branches in the disk in possible through Theorem \ref{thm:exist}.
\end{rem}

Having built solutions to \eqref{BS}, one may now consider the stability of the corresponding bound-state under the complex Ginzburg-Landau flow. Fix an eigenvalue $\tilde{\lambda}$ and a bifurcation branch $(\lambda(\epsilon), u(\epsilon))$ starting at $(\tilde{\lambda},0)$ given either by Theorem \ref{thm:exist} (in the multiple case) or by Remark \ref{rem:simples} (in the simple case). Set $k(\epsilon)-i\omega(\epsilon):=e^{i\theta}\lambda(\epsilon)$. Then, as already mentioned, $v=e^{i\omega(\epsilon)t}u(\epsilon)$ solves \eqref{CGL} for $k=k(\epsilon)$.

\begin{thm}\label{thm:inst} Let $\Omega\subset \R^N$ be a bounded domain and $1 \leq \sigma \le 2/(N-2)^+$. Suppose that $\tilde{\lambda}$ is an eigenvalue of the Dirichlet-Laplace operator and it is not the first eigenvalue. 
	Then, for  $\varepsilon$ sufficiently small, the standing wave $u$ is orbitally unstable.
\end{thm}

\begin{rem}
	The condition $\sigma\le 2/(N-2)^+$ is only required in order to guarantee the local well-posedness of \eqref{CGL} (see \cite{Henry}).
\end{rem}

We also take this opportunity to present several examples that give some insight concerning the relationship between reflection symmetries and bifurcation branches exploited in \cite{DelPinoGarciaMusso}. As we will see, the equivalence found in \cite{DelPinoGarciaMusso} between the reflection symmetries of the square and the bifurcation branches starting at the second eigenvalue is a very particular feature and does not hold in general.

\vskip15pt\noindent \textbf{Structure of the article.} In Section 2, we present the proofs of the bifurcation results (Theorems \ref{thm:exist}, \ref{thm:caract} and their corollaries). In Section 3, we prove the instability result (Theorem \ref{thm:inst}). We conclude with a discussion on the connection between reflection symmetries and real bifurcation branches (Section 4).

\section{Bifurcation analysis}

%
%  Regarding the existence of standing
%  wave solutions, some partial results were obtained in the case of a bounded
%  domain; cf.\cite{Caz-Weiss2}, \cite{Puel}, \cite{C-F1} and \cite{C-F2}. In this last paper, the authors,
%  through an argument of bifurcation for a double eigenvalue of the Dirichlet-Laplace
%  operator and using the Lyapunov-Schmidt reduction, construct small standing waves, i.e.
%  periodic solutions $u = e^{i\omega t} \phi$ of  \eqref{CGL},
%  $\omega \in \R$ and $\phi\in H^1_0(\Omega)$ a solution of the stationary equation
%   \begin{equation}
%  	i \omega \phi = e^{i \theta} \Delta \phi + e^{i \gamma} |\phi|^\sigma \phi + k \phi
%  \end{equation}
%which we rewrite in the more convenient way
%  \begin{equation}\tag{BS}\label{BS}
%  	\lambda \phi + \Delta \phi = \eta|\phi|^\sigma \phi,\quad \lambda, \eta \in \C.
%  	%,\quad \lambda:=(k-i\omega)e^{-i\theta}, \ f(x,\phi) :=-e^{i\gamma-\theta}|\phi|^\sigma \phi.
%  \end{equation}

\begin{defi}\label{defi:bifurc}
	 We say that a continuous mapping $(\lambda,u);[0,\delta)\to \C\times H^1_0(\Omega)$ is a bifurcation branch starting from $(\tilde{\lambda},0)$ if $(\lambda(0),u(0))=(\tilde{\lambda},0)$ and, for each $\epsilon\in (0,\delta)$, $u(\epsilon)$ is a nontrivial solution to \eqref{BS} with $\lambda=\lambda(\epsilon)$.
\end{defi}
\begin{defi}\label{defi:gauge}
	Equation \eqref{BS} is gauge invariant: if $u$ is a solution, so is $zu$, for any $z\in \mathbb{S}^1:=\{z\in\C: |z|=1\}$. If $(\lambda_1,u_1)$ and $(\lambda_2,u_2)$ are two bifurcation branches defined on $[0,\delta)$ and, for each $\epsilon\in[0,\delta)$, there exists $z(\epsilon)\in \mathbb{S}^1$ such that $u_1(\epsilon)=z(\epsilon)u_2(\epsilon)$, we say that the branches are the same modulo gauge invariance.
\end{defi}

	Define $L=-\Delta$ and $ M u  =  -\eta | u |^\sigma  u  $ and
	$$
	(u,v)_{L^2}=\int_\Omega u\overline{v}dx.
	$$ Represent by $V$ the corresponding
	eigenspace, spanned by $u_1,\ldots, u_p$. Taking  
	$P : L^2(\Omega) \rightarrow V^\perp$ and applying the Lyapunov-Schmidt reduction, equation \eqref{BS}
	is equivalent to the system 
	\begin{align}
		P(\lambda  u  - L  u  + M  u ) &  = 0 \label{eq1} \\
		(\lambda  u  - L  u  + M  u  , u_j)_{L^2} & = 0, \,\, j=1, \ldots, p \label{eq2}
	\end{align}
	We look for solutions of this system in the form
	$$  u  = y + \sum_{j=1}^p \epsilon_j u_j, \quad y \in V^\perp $$
	where, without loss of generality, we may assume $\epsilon_1 > 0$ (due to the gauge invariance of the problem). From \eqref{eq1},
	\begin{equation}\label{equ-y}
		y = (\lambda-PL)^{-1} \Big[-P M ( y + \sum_{j=1}^p \epsilon_j u_j) \Big] 
	\end{equation}
	and
	\begin{equation}\label{equ2}
		\epsilon_m (\lambda - e^{i\theta} \tilde\lambda) = - \Big(M (y + \sum_{j=1}^p \epsilon_j u_j), u_m \Big)_{L^2},
		\quad m=1, \ldots, p.
	\end{equation}

Following exactly the same proof as in \cite[Lemma 3.1]{MS_GL2}, 

\begin{lem}\label{lem:construcao_y}
	There exists $\delta_0>0$ such that,  for all $0<\delta<\delta_0$, given $\epsilon_1,\dots,\epsilon_p,\lambda\in \C$ with
\begin{equation}\label{eq:pequenez}
	|\epsilon_j|< \delta,\ j=1,\dots, p,\quad |\lambda-\tilde{\lambda}|<\delta,
\end{equation}
	there exists a unique solution 
	$ y = y(\epsilon_1,\dots,\epsilon_p,\lambda) \in H^1_0(\Omega)$
	of \eqref{equ-y}. Moreover, for some universal constants $C > 0, K > 0$,
	\begin{equation}\label{estim-y}
		\| y(\epsilon_1,\dots,\epsilon_p,\lambda)\|_{H^1_0} \leq C \delta^{{\sigma}+1}  \\
	\end{equation}
	and
	\begin{equation}\label{estim-lip}
		\|  y(\epsilon_1,\dots,\epsilon_p,\lambda) -  y(\epsilon_1',\dots,\epsilon_p',\lambda') \|_{H^1_0} \leq K\delta^{{\sigma+1}}| (\epsilon_1,\dots,\epsilon_p,\lambda) - (\epsilon_1',\dots,\epsilon_p',\lambda') | \\
	\end{equation}
	whenever  $(\epsilon_1,\dots,\epsilon_p,\lambda)$ and $ (\epsilon_1',\dots,\epsilon_p',\lambda') $ satisfy \eqref{eq:pequenez}.
	%	
	%	\noindent In addition, for $\lambda \in V_r$, there exists $\tilde K > 0$, such that
	%	\begin{equation}\label{estim-lip2}
		%	\| y_{\varepsilon_1} (\lambda) - y_{\varepsilon_2} (\lambda) \| \leq \tilde K |\varepsilon_1 - \varepsilon_2|
		%	\end{equation}
	%	for all  $0< \varepsilon_1 , \varepsilon_2 \leq \varepsilon_0 $.
	%	
\end{lem}

	Setting $\alpha_j= \epsilon_j/ \epsilon_1$ and taking $m=1$ in \eqref{equ2},
	\begin{equation}\label{equ3}
		\begin{split}
			\lambda - e^{i\theta} \tilde\lambda & = - \frac{1}{\epsilon_1} \Big(M (y + \sum_{j=1}^p \epsilon_j u_j), u_1\Big)_{L^2} \\
			& = - \epsilon_1^\sigma \Big( M (u_1+  \sum_{j=2}^p \alpha_j u_j), u_1\Big)_{L^2} +  
			Q_1 (y, \epsilon_1, \alpha_2, \ldots,\alpha_p) \\
		\end{split}
	\end{equation}
	with
	\begin{equation}
		Q_1(y, \epsilon_1,\A) = -\frac{1}{\epsilon_1} \Big(M(y + \sum_{j=1}^p \epsilon_j u_j) - M(\sum_{j=1}^p \epsilon_j u_j), u_1\Big)_{L^2}.
	\end{equation}
	On the other hand, \eqref{equ2} also implies
	\begin{equation}\label{equ4}
		\begin{split}
			\alpha_m\Big(M (y + \sum_{j=1}^p \epsilon_j u_j), u_1\Big)_{L^2} &= -\alpha_m \epsilon_1 (\lambda - e^{i\theta} \tilde\lambda)
			= -\epsilon_m (\lambda - e^{i\theta} \tilde\lambda) \\
			& = \Big( M (y + \sum_{j=1}^p \epsilon_j u_j), u_m\Big)_{L^2}. \\
		\end{split}
	\end{equation}
	Setting
	\begin{equation}
		Q_m(y, \epsilon_1,\A)  = 
		\frac{1}{\epsilon_1^{\sigma+1}}
		\Big(M(y + \sum_{j=1}^p \epsilon_j u_j) - M(\sum_{j=1}^p \epsilon_j u_j), \overline{\alpha_m} u_1 - u_m \Big)_{L^2} ,
	\end{equation}
	it follows from \eqref{equ4} that 
	\begin{equation}\label{equ5}
		\Big(  M (u_1+  \sum_{j=2}^p \alpha_j u_j), \overline{\alpha_m} u_1 - u_m \Big)_{L^2} + Q_m = 0,\,\, m=2, \ldots, p.
	\end{equation}

	\begin{proof}[Proof of Theorem \ref{thm:exist}]
		From the above discussion and Lemma \ref{lem:construcao_y}, it suffices to prove that, setting $y=y(\epsilon_1,\dots, \epsilon_p,\lambda)$ given by Lemma \ref{lem:construcao_y}, equations \eqref{equ3} and \eqref{equ5} define $\lambda$ and $\alpha_j$ implicitly as function of $\epsilon_1$. First, observe that, if one drops the remainder terms $Q_m$, the system
		reduces to
		\begin{equation}\label{system}
			\begin{split}
				F_1(\epsilon_1,\lambda, \alpha_2, \ldots,\alpha_p) &=   \lambda -  \tilde\lambda +
				\epsilon_1^\sigma \Big( M (u_1+  \sum_{j=2}^p \alpha_j u_j), u_1\Big)_{L^2} = 0 \\
				F_m(\epsilon_1,\lambda, \alpha_2, \ldots,\alpha_p) &= P_m(\alpha_2, \ldots,\alpha_p) = 0,\,\, m=2, \ldots, p.
			\end{split}
		\end{equation}
		 This system satisfies the conditions 
		of the Implicit Function Theorem at $(\epsilon_1, \lambda, \A) = (0, \tilde \lambda,
		\A^0)$, since 
		$$ \frac{\partial (F_1,\ldots, F_p)}{\partial(\lambda, \alpha_2,\ldots,\alpha_p)}(\tilde{\lambda}~,\A^0) =
		\frac{\partial (P_2, \ldots, P_p)}{\partial( \alpha_2,\ldots,\alpha_p)} (\A^0) \neq0. $$
		Using the estimates \eqref{estim-y} and \eqref{estim-lip}, one may prove that $Q_m$, $m=1,\dots,p$ are Lipschitz continuous in $\epsilon_1, \lambda$ and $\alpha_m, m=2,\dots, p$, with  an arbitrarily small constant (see \cite[Proof of Theorem 1.4]{MS_GL2}). 
		Therefore  the system  \{\eqref{equ3},\eqref{equ5}\} is a small Lipschitz perturbation of \eqref{system} and one may apply the Implicit Function Theorem for Lipschitz functions (\cite[Section 7.1]{Clarke}) to finish the proof.
\end{proof}

\begin{proof}[Proof of Theorem \ref{thm:caract}] The proof follows closely the argument of \cite[Lemma 2.1]{DelPinoGarciaMusso}. 
	Take $\epsilon_n\to 0$ and write $\lambda_n=\lambda(\epsilon_n)$, $u_n=u(\epsilon_n)$, so that $ \lambda_n \rightarrow  \tilde\lambda$
	and $  u _n \rightarrow 0$ in $ C(\overline\Omega) $. After normalization
	$\tilde u_n = u_n/\|u_n\|_\infty$, we obtain
	\begin{equation}\label{sistematilde}
		\begin{cases} 
			\lambda_n \tilde u_n +  \Delta \tilde u_n + | u_n|^{\sigma}\tilde u_n = 0 \\
			\tilde u_n = 0 \quad{\rm on}\quad \partial\Omega
		\end{cases}
	\end{equation}

	Since $\tilde{u}_n$ is bounded in $C(\overline{\Omega})$, by compactness, $\tilde u_n \rightarrow u_0$ in $ C(\overline\Omega)$ and $\tilde{\lambda}u_0+\Delta u_0=0$.
	Moreover, since $\|\tilde u_n\|_{L^\infty(\Omega)}=1,$ $u\notequiv 0$. This implies that, up to reordering of $u_1,\dots, u_p$,
	$$
	u_0=c(u_1+\alpha_2u_2 +  \dots + \alpha_pu_p)=:cu_{\A},\quad c,\alpha_2,\dots, \alpha_m \in \C.
	$$
		Decomposing $\tilde u_n = \chi_n + \psi_n$
	with $\chi_n\in V, \psi_n \in V^\perp$, 
	\begin{equation}\label{sistemapsi}
			\tilde\lambda \psi_n  + \Delta \psi_n + (\lambda_n - \tilde\lambda) (\chi_n + \psi_n) 
			+     \|u_n\|_\infty^\sigma | \chi_n + \psi_n |^{\sigma} ( \chi_n + \psi_n)  = 0.
		\end{equation} 
	Multiplying by $\bar u_{\A}$ and integrating,
	$$ \frac{\lambda_n-\tilde\lambda}{\|u_n\|_\infty^\sigma} \int_\Omega \chi_n \bar u_{\A} dx +
	\int_\Omega |\chi_n + \psi_n |^{\sigma}(\chi_n + \psi_n)  \bar u_{\A} dx = 0 $$
	Therefore
	\begin{equation}\label{rho}
		\lim_{n\rightarrow\infty} \frac{\lambda_n-\tilde\lambda}{\|u_n\|_\infty^\sigma}  = - c^\sigma 
		\frac{\int_\Omega |u_{\A}|^{\sigma+2} dx }{\int_\Omega |u_{\A}|^2 dx} 
	\end{equation} 
Consider the elements 
	$$ v_{\alpha_m} := (\overline{\alpha_m} u_1 - u_m),\quad m = 2,\ldots,p $$
	and remark that $(u_{\A}, v_{\alpha_m})_{L^2}  = 0,\,\, m=2,\ldots,p$.
We multiply  \eqref{sistemapsi}
	by $\bar v_{\alpha_m}$, divide by $\|u_n\|_{L^\infty}^\sigma$, integrate over $\Omega$ to find
	$$
	\frac{\lambda_n-\tilde{\lambda}}{\|u_n\|_{\infty}^\sigma}\int_\Omega\psi_n\overline{v_{\alpha_m}} + \int_\Omega |\tilde{u}_n|^\sigma\tilde{u}_n \overline{v_{\alpha_m}} = 0
	$$
	In the limit, we find
	\begin{equation}\label{alphacondition}
		P_m(\alpha_2,\dots, \alpha_m)=\int_{\Omega}  |u_{\A}|^\sigma u_{\A} (\alpha_mu_1 - u_m) = 0, \quad m=2,\ldots,p
	\end{equation}
	and $\A=(\alpha_2,\ldots, \alpha_p)$ is a solution to \eqref{polinomios}.

\end{proof}

\begin{proof}[Proof of Corollary \ref{cor:equiv}] 
	\textit{Step 1}. As in the previous proof, if $\epsilon_n\to 0$, up to a subsequence,
	$$
	u(\epsilon_n)=c_n\|u(\epsilon_n)\|_{\infty}(u_1+\alpha_{2,n}u_2+\dots + \alpha_{p,n}u_p) + y_n,\quad y_n\in V^\perp,\  c_n\to c\neq0,
	$$
	and
	$$
	(\alpha_{2,n}, \dots, \alpha_{p,n}) \to \A \mbox{ solution to \eqref{polinomios}}.
	$$
	Using the gauge invariance, we may assume that $c_n\in \R^+$. Since $\|u(\epsilon_n)\|_{\infty}\to 0$ and $\lambda(\epsilon_n)\to \tilde{\lambda}$, Lemma \ref{lem:construcao_y} implies that, for $n$ large enough, $y_n$ is uniquely determined by the values of
	$$
	c_n\|u(\epsilon_n)\|_{\infty},\ \alpha_{2,n}, \dots,\alpha_{p,n}\ \mbox{and }\lambda(\epsilon_n). 
	$$
	On the other hand, the proof of Theorem \ref{thm:exist} implies that $c_n\|u(\epsilon_n)\|_{\infty}$ also determines uniquely the values of $\alpha_{2,n}, \dots,\alpha_{p,n}$ and $\lambda(\epsilon_n)$ in a neighbourhood of $(0,\tilde{\lambda},\A)$. Therefore, denoting by $(\lambda',u')$ by bifurcation branch built in Theorem \ref{thm:exist} starting from $\A$,  $$(\lambda(\epsilon_n),u(\epsilon_n))=(\lambda'(c_n\|u_n\|_\infty), u'(c_n\|u_n\|_\infty)).$$
	
	\textit{Step 2.} Suppose that there exists a sequence $\epsilon_n\to 0$ for which the claimed result does not hold. By Step 1, there is a subsequence which does satisfy Corollary \ref{cor:equiv}, leading to a contradiction.
\end{proof}

\begin{proof}[Proof of Corollary \ref{cor:cubo}]
To reduce system \eqref{polinomios}, one computes explicitly the coefficients (as it was done in \cite{MugnaiPistoia}). Write $\Omega=\prod_{j=1}^N (0,L_j)$ and define $\phi_{k,L}(x)=\sin(k\pi x/L)$, $k\in \N$, $L>0$. Then
$$
\int_0^L \phi_{k,L}(x)^4 dx = \frac{3L}{8}, \quad \int_0^L \phi_{k,L}(x)^2\phi_{j,L}(x)^2 dx = \frac{L}{4}, \ j\neq k.
$$
%and
%$$
%\int_0^L \phi_{k,L}^2(x)\phi_{j,L}(x)\phi_{l,L}(x)dx = 0,\quad j\neq l.
%$$
The eigenspace associated to $\tilde{\lambda}$ is generated by $u_1,\dots, u_p$ of the form
$$
\prod_{j=1}^N \phi_{k_j,L_j}(x_j),\quad\mbox{ where }(k_1,\dots, k_N)\in \N^N\mbox{ satisfies} \sum_{j=1}^N \frac{\pi^2 k_j^2}{L_j^2}= \tilde{\lambda}.
$$
Then, for $1\le m \le p$,
$$
\int_\Omega u_m^4(x)dx = \prod_{j=1}^N \frac{3L_j}{8}=: A,\quad \int_\Omega u_m^2u_k^2(x)dx = \prod_{j=1}^N \frac{L_j}{4}=: B,\ k\neq m.
$$
Following \cite[Proof of Theorem 1.1]{MugnaiPistoia}, 
$$
\int_\Omega u_m^2(x)u_k(x)u_l(x)dx =0,\quad k\neq l.
$$
Finally, by \eqref{eq:hipotese4}, if $p\ge 4$ and $k,m,l,h$ are four different indices,
$$
\int_\Omega u_m(x)u_k(x)u_l(x)u_h(x)dx=0.
$$
This allows us to compute
\begin{align*}
P_m(\alpha_2,\dots,\alpha_p)&=\int_\Omega
\left|u_1+\sum_{ j=2}^p \alpha_j u_j\right|^2 \left(u_1+\sum_{ j=2}^p \alpha_j u_j\right) 
{(\alpha_m u_1 - u_m)} dx\\& = (B-A)(|\alpha_m^2|-1)\alpha_m - 2B(\alpha_m^2-1)\Re\alpha_m.
\end{align*}
Hence the real solutions to \eqref{polinomios} are given by
$$
 \alpha_m \in \{0,\pm 1\}, \quad m=2,\dots, p.
$$
A simple computation shows these solutions are nondegenerate. Thus the number of bifurcations with a nonzero $u_1$-coefficient is $3^{p-1}$. 

Next, if one looks for the branches where the $u_1$-coefficient is zero and the $u_2$-coefficient is non-zero, the same computations yield $3^{p-2}$ bifurcation branches. Iterating this procedure, the total number of bifurcations is
$$
3^{p-1}+3^{p-2}+\dots + 1 = \frac{3^p-1}{2}.
$$ 
\end{proof}

\begin{rem} Fix $\Omega=(0,1)^2$. A consequence of Corollary \ref{cor:cubo} is that, for cubic nonlinearities and under \eqref{eq:hipotese4}, the real bifurcation branches correspond to $(\alpha_2^0,\alpha_3^0)\in \{-1,0,1\}^2$.  One may ask if the same happens for other nonlinearities. In the case of the quintic nonlinearity $\sigma =4$, the first triple eigenvalue has eigenspace generated by
$$
u_1(x,y)=\sin(\pi x)\sin (7\pi y),\quad u_2(x,y)=u_1(y,x),\quad u_3(x,y)=\sin(5\pi x)\sin(5\pi y).
$$
%From symmetry arguments applied to the first two eigenfunctions, one can already find four bifurcation branches.
There are nine bifurcations of the form
$$
u_3 + \alpha_2 u_1 + \alpha_3 u_2
$$
for $(\alpha_2,\alpha_3)$ as shown in Figure \ref{fig:sigma4}. It becomes clear that the set of solutions to \eqref{polinomios} becomes nontrivial for higher order nonlinearities.
\newpage
\end{rem}
\begin{figure}[h]
	\centering
	\includegraphics[width=.3\linewidth]{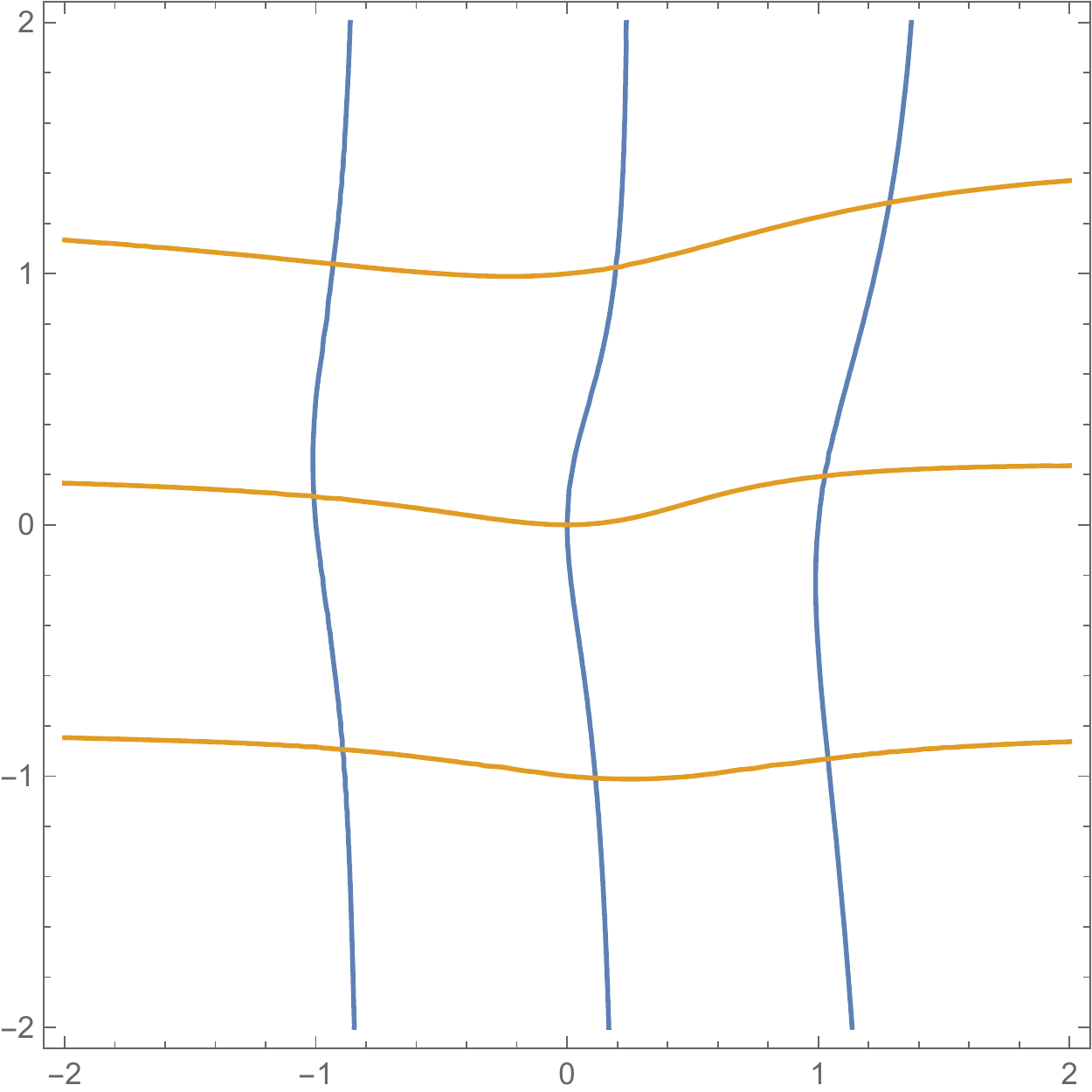}
	%			\caption{Two $\alpha$'s equal to 0\\ 3 branches}
	\caption{The nodal sets for $P_2$ and $P_3$. The nine bifurcations for which the $u_3$-component is nontrivial correspond to the nine intersection points.} 
	\label{fig:sigma4}
\end{figure}

\section{Stability of bound-states under the (CGL) flow}

To prove Theorem \ref{thm:inst}, we perform an analysis on the spectrum of the period map for the linearized equation
\begin{equation}\label{eq-variational}
	v_t = e^{i\theta} \Delta v  + k v + B(t) v
\end{equation}
with
\begin{equation}\label{operador-B}
	B(t) v = e^{i \gamma} 
	\left( \frac{\sigma+2}{2} | u(t)|^\sigma v + \frac{\sigma}{2} u(t)^{1+\sigma/2} {\bar u(t)}^{\sigma/2-1} \bar v\right) 
\end{equation} 
%We have the following result:
%Let $ u= e^{i \omega t} \phi$ the standing wave of \eqref{CGL} with
%$ \omega = \Im \lambda (\varepsilon), k = \Re \lambda (\varepsilon)$,
%\begin{equation}\label{lambdaepsilon}
%	\lambda (\varepsilon) = e^{i \theta} \lambda_1 - e^{i \gamma} \varepsilon^\sigma
%	\int_\Omega |u_1|^\sigma u^2 dx + o(\varepsilon^\sigma) 
%\end{equation}
%$\lambda_1$ the first eigenvalue of the Dirichlet-Laplace operator.
Define the evolution operator $R(t, s)$ for the equation \eqref{eq-variational} as
$$ R(t, s) v_0 = v(t; s, v_0), $$
where $v$ is the solution of \eqref{eq-variational} with initial data $v(s)=v_0$. We now define the period map as the linear operator
$U_0 = R(T, 0), \, T = 2 \pi/\omega$.
%, will be analysed below.

\begin{proof}[Proof of the Theorem \ref{thm:inst}] 
	Throughout the proof, it is convenient consider real spaces composed 
	of complex valued functions. We represent them using bold face characters.
	\vskip10pt
	
%	\noindent\textit{Proof of part 1. }
	Represent by $S(t) = e^{A t}$ the  semigroup generated by the operator 
	$A := A(\varepsilon) = e^{i\theta} \Delta + k\,$,
	$ k= \Re e^{i\theta}\lambda(\varepsilon)$.
	It is well-known that $S(t)$ is an analytic semigroup in ${\bf L}^2(\Omega)$ (see, e.g., \cite[Theorem 2.7, pg. 211]{Pazy}). Since $\|u\|_{{\bf H}^1_0(\Omega)}=O(\epsilon)$, a standard bootstrap argument yields $\|u\|_{{\bf L}^\infty(\Omega)}\lesssim O(\epsilon)$. Thus
	the linear operator $B(t)$ referred in \eqref{operador-B} satisfies
	\begin{align}\label{estim-B}
		\| B(t) \|_{\mathcal{L} (\bf{L}^2(\Omega))} &\leq \left\|	 \frac{\sigma+2}{2} | u(t)|^\sigma + \frac{\sigma}{2} u(t)^{1+\sigma/2} {\bar u(t)}^{\sigma/2-1} \right\|_{\bf{L}^\infty(\Omega)}\\&\leq (\sigma +1) \|u\|_{\bf{L}^\infty(\Omega)}^\sigma \leq
		(\sigma +1) C \varepsilon^\sigma
	\end{align}
	for all $t\in [0, T]$ and for some universal constant $C$.
	Consider now the linear evolution equation
	$$ v_t = A v + B(t) v.$$
	Given an initial data $v_0 \in {\bf L}^2(\Omega)$,
	\begin{equation}\label{eq-integral}
		R(t,0) v_0 = v(t) = S(t) v_0 + \int_0^t S(t-s) B(s) v(s) ds 
	\end{equation}
	and
	\begin{equation}\label{estim-v}
			\|v(t)\|_{{\bf L}^2(\Omega)}  \leq C \|v_0\|_{{\bf L}^2(\Omega)} + \int_0^t C (\sigma + 1) \varepsilon^\sigma
			\|v(s)\|_{{\bf L}^2(\Omega)} ds 
	\end{equation}
	for all $ t\in [0, T]$. Using Gronwall's inequality, we conclude that
\begin{equation}
	\|v(t)\|_{{\bf L}^2(\Omega)}  \leq C e^{CT}  \|v_0\|_{{\bf L}^2(\Omega)},\quad t\in [0,T].
\end{equation}
	 It follows from \eqref{estim-B}, \eqref{eq-integral} and \eqref{estim-v} 
	\begin{equation}\label{periodic map}
		U_0 = S(T) + K(\varepsilon)
	\end{equation}
	with 
	\begin{equation}\label{estim-K}
		K(\varepsilon) = \int_0^T S(t-s) B(s) v(s) ds, \quad
		\|  K(\varepsilon) \|_{ \mathcal{L} ({\bf L}^2(\Omega))} \leq C \varepsilon^\sigma 
	\end{equation}
	On the other hand, the spectrum of the operator $A$, $\Sigma (A)$, verifies
	\begin{equation}
			 \Re \, \Sigma (A) := \{\Re \lambda : \lambda \in \Sigma (A)\} 
			= \left\{ - \cos \theta (\lambda_j-\tilde{\lambda}) + O(\varepsilon^\sigma) \right\}_{j\ge1}
	\end{equation} 
	where $\lambda_1<\lambda_2\le \lambda_3 \le \dots $ are the eigenvalues of the Dirichlet-Laplace operator. 
	Since $S(t)$ is an analytic semigroup,
	the spectral mapping theorem establishes $\Sigma(S(T)) = e^{T \Sigma (A)}$ (\cite{Nagel}, pag. 281). Thus
	$$ \{ |\lambda| : \lambda \in \Sigma (S(T))\} = \{
	e^{-T \cos \theta (\lambda_j-\tilde{\lambda}) + O(\varepsilon^\sigma)} \}_{j \ge 1}. $$
	From \eqref{periodic map} and \eqref{estim-K}, the period map $U_0$ is a small perturbation of $S(T)$, which implies the existence of an eigenvalue $\mu\in\sigma(U_0)$ with $|\mu|>1$. Since we are working on a real Banach space and $S(t)$ is an analytic semigroup, the result now follows from \cite[Th. 8.2.4]{Henry}. 

\end{proof}

\section{Real bifurcations and reflection symmetries}
\vskip10pt
To conclude, we now discuss the connection between reflection symmetries of $\Omega$ and real bifurcation branches. A reflection symmetry is any reflection mapping with respect to a hyperplane which leaves $\Omega$ invariant. In particular, the hyperplane (or symmetry axis) splits $\Omega$ into two congruent domains $\Omega_1$ and $\Omega_2$ which are reflections of one another. For a given eigenvalue $\lambda$, suppose that there exists an eigenfunction which vanishes along the symmetry axis and is positive in $\Omega_1$. This implies in particular that $\lambda$ is the first eigenvalue in $\Omega_1$ and so, being simple, there exists a unique bifurcation branch when the problem is restricted to $\Omega_1$. Through an odd extension of the bifurcation branch to the whole $\Omega$, we then obtain a bifurcation branch for our initial problem.

In the case of the second eigenvalue of the square, the corresponding (two-dimensional) eigenspace includes four elements whose nodal lines are precisely the four reflection axii of the square:
\begin{figure}[h]
	\centering
	\begin{subfigure}{.2\textwidth}
		\centering
		\includegraphics[width=.6\linewidth]{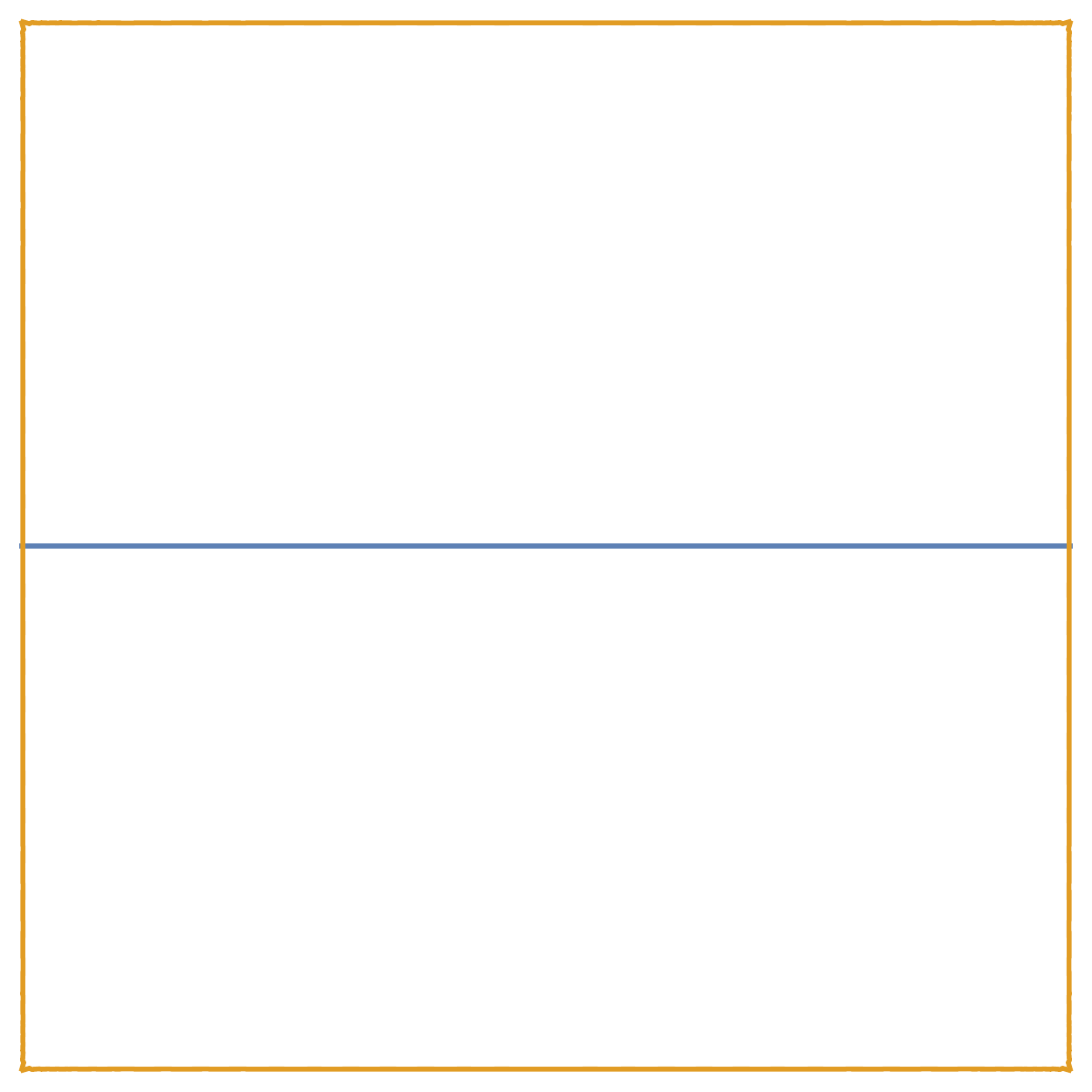}
	\end{subfigure}%
	\begin{subfigure}{.2\textwidth}
		\centering
		\includegraphics[width=.6\linewidth]{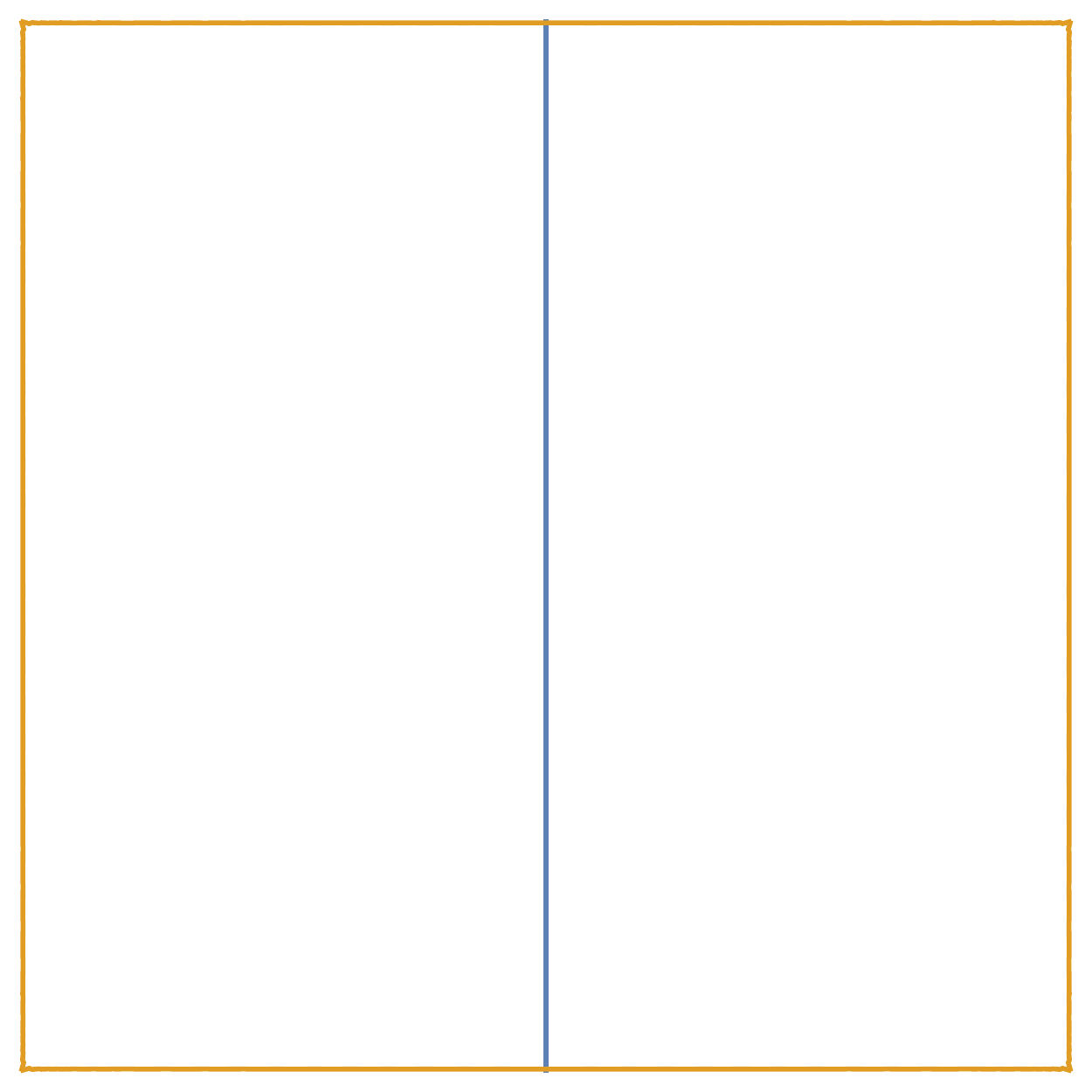}
	\end{subfigure}
	\begin{subfigure}{.2\textwidth}
		\centering
		\includegraphics[width=.6\linewidth]{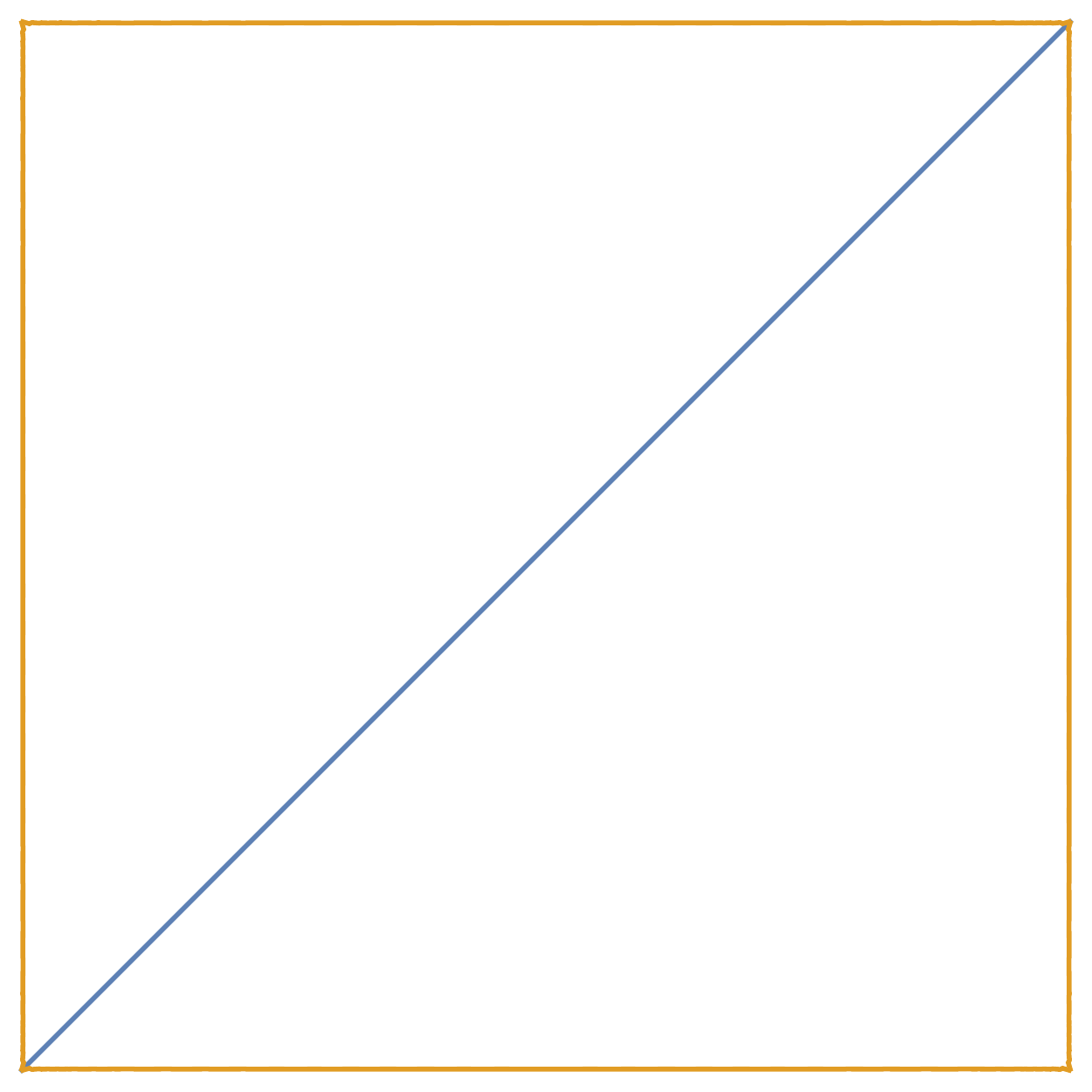}
	\end{subfigure}
	\begin{subfigure}{.2\textwidth}
		\centering
		\includegraphics[width=.6\linewidth]{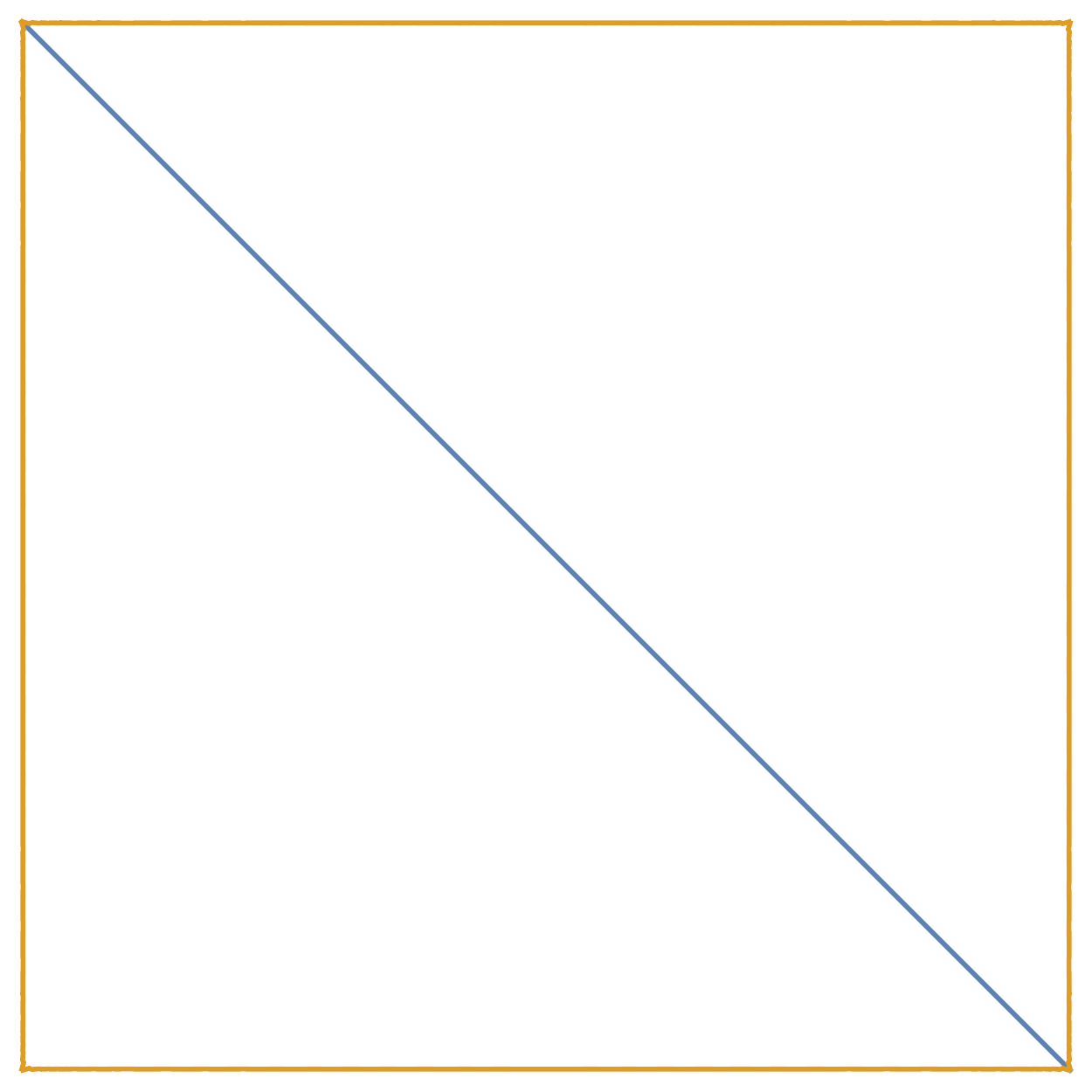}
	\end{subfigure}
	\caption{Nodal lines of four eigenfunctions coincide with the four reflection symmetry axii.}
	\label{fig:1}
\end{figure}
\vskip10pt
From Corollary \ref{cor:cubo}, the number of bifurcation branches is $\frac{3^2-1}{2}=4$. Thus we conclude that all bifurcation branches arise from symmetries, as observed in \cite{DelPinoGarciaMusso}. 
This equivalence between bifurcations and symmetries does not hold in general, as one can see in the following examples concerning the bifurcation problem for $\sigma=2$:
\begin{enumerate}
	\item (Higher dimension) For the second eigenvalue of the cube, Corollary \ref{cor:cubo} implies the existence of exactly thirteen bifurcation branches, which one may separate into three groups, depending on whether $\alpha_2^0$ and $\alpha_3^0$ are zero or not. The nodal sets are plotted in Figure \ref{fig:2}.	It is clear that the last four branches do not arise from any kind of reflection symmetry.
	\begin{figure}[h]
		\centering
		\begin{subfigure}{.3\textwidth}
			\centering
			\includegraphics[width=.8\linewidth]{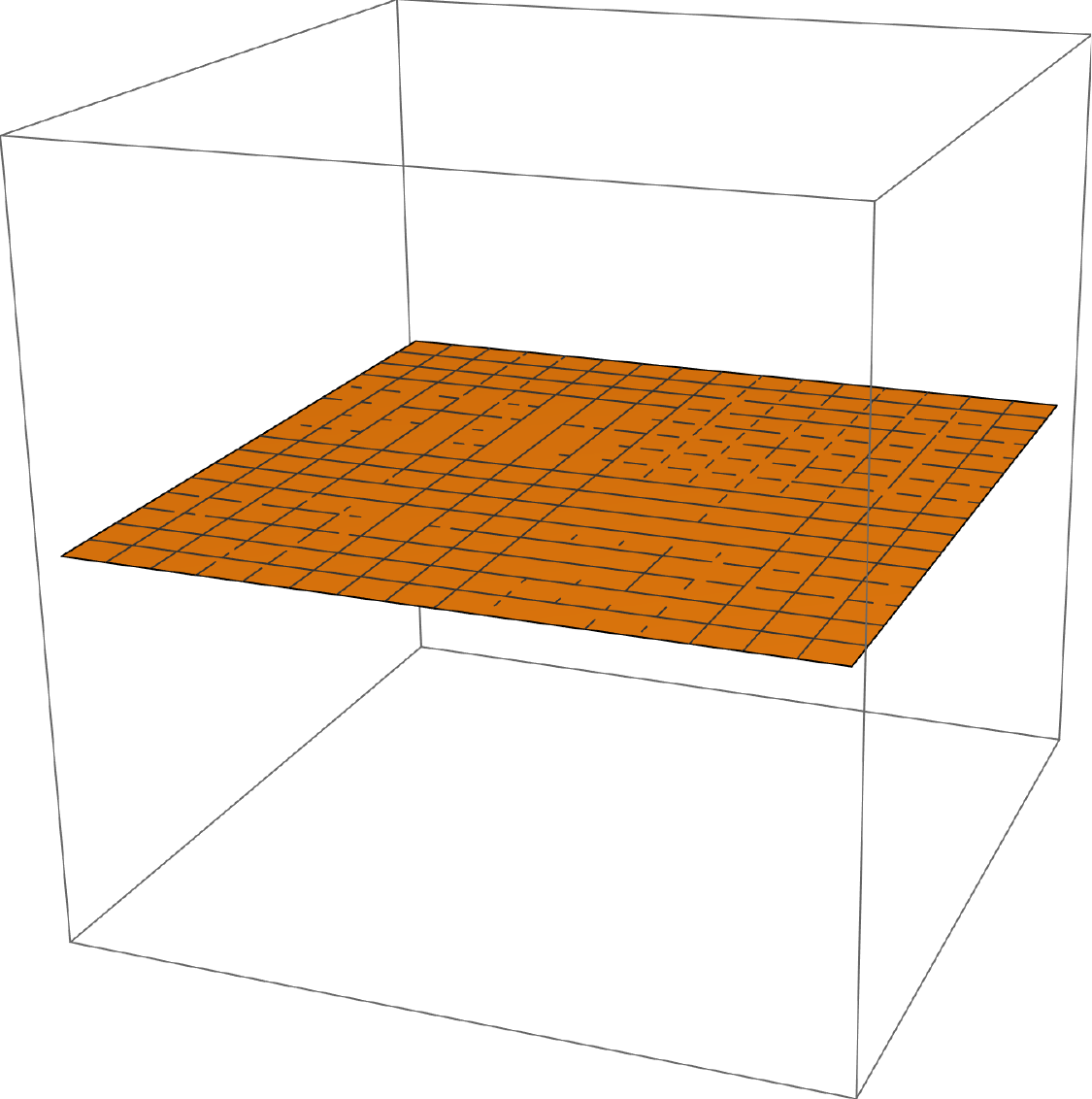}
			%			\caption{Two $\alpha$'s equal to 0\\ 3 branches}
			\caption{ \tabular[t]{@{}l@{}}Two $\alpha$'s equal to 0\\-  3 branches\endtabular}
			
		\end{subfigure}%
		\begin{subfigure}{.3\textwidth}
			\centering
			\includegraphics[width=.8\linewidth]{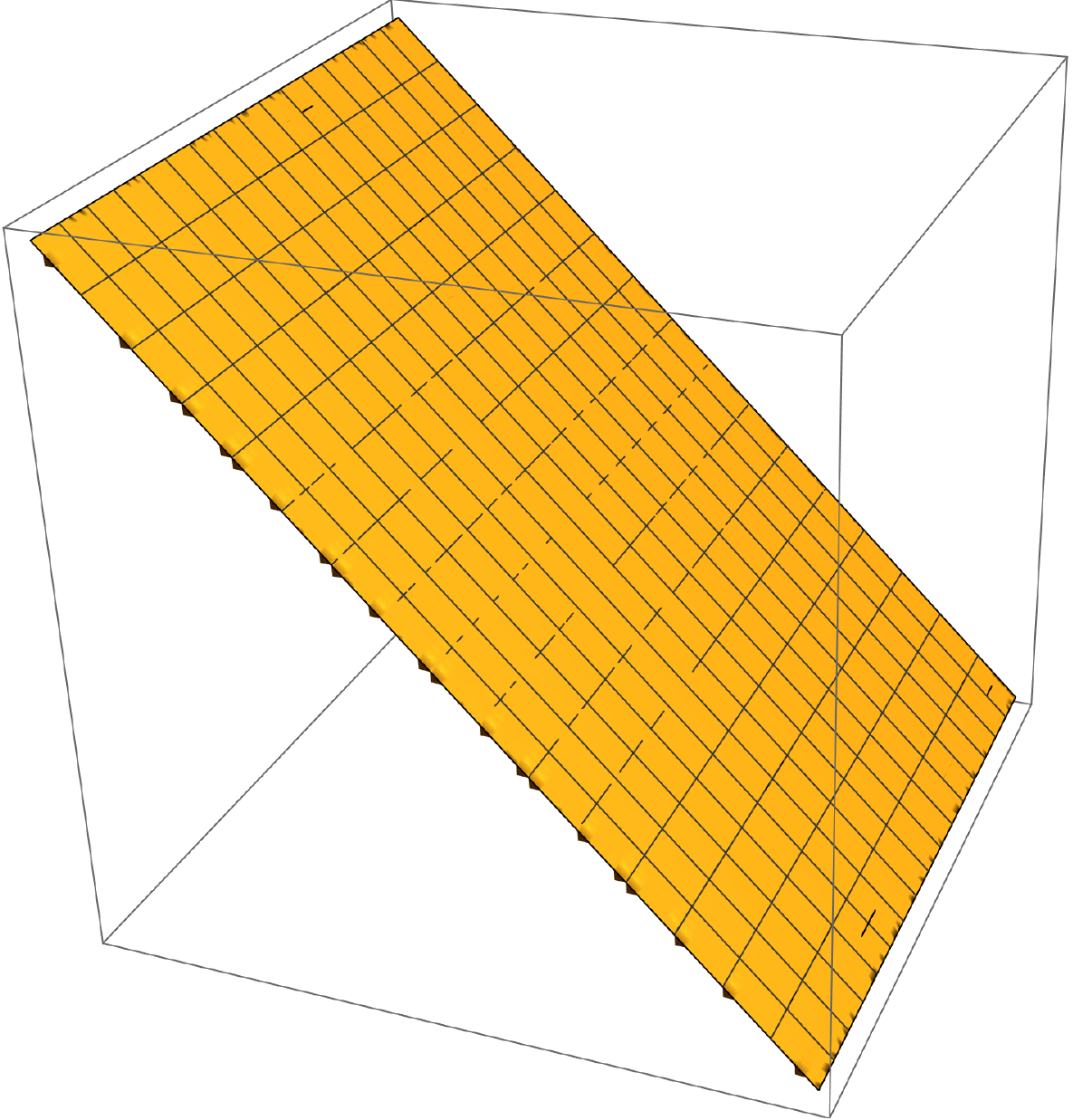}
			\caption{ \tabular[t]{@{}l@{}}One $\alpha$ equal to 0\\-  6 branches\endtabular}
		\end{subfigure}
		\begin{subfigure}{.3\textwidth}
			\centering
			\includegraphics[width=.8\linewidth]{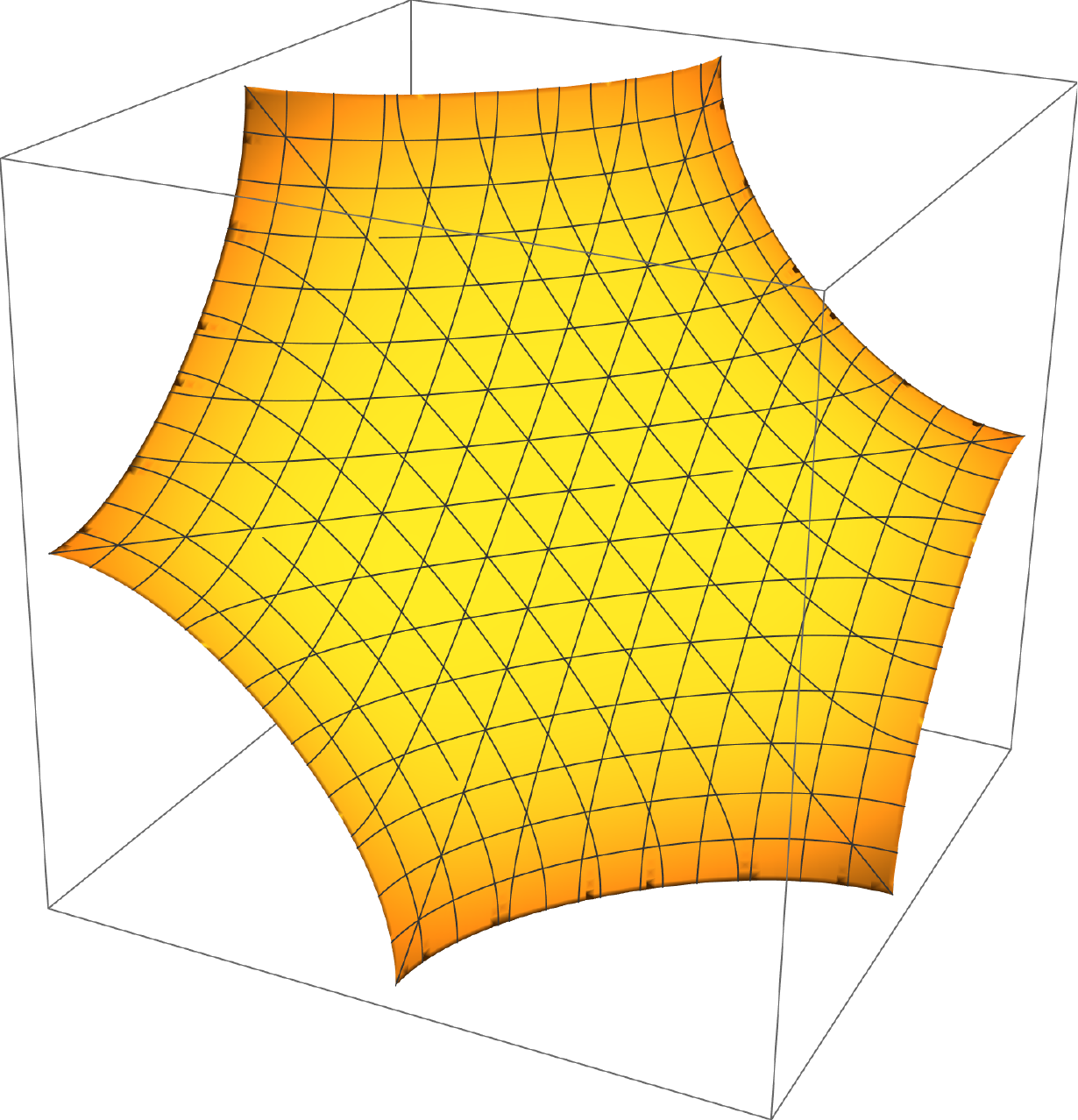}
			\caption{ \tabular[t]{@{}l@{}}Both $\alpha$'s nonzero\\-  4 branches\endtabular}
		\end{subfigure}
		\caption{The three different types of bifurcations for the second eigenvalue in the cube.}
		\label{fig:2}
	\end{figure}
	\item (Higher eigenvalue) In the square $(0,1)^2$, the second multiple eigenvalue is the third one, with eigenspace generated by
	$$
	u_1(x,y)=\sin(\pi x)\sin(3\pi y),\quad u_2(x,y)=\sin(3\pi x)\sin(\pi y),
	$$
	The corresponding nodal lines are two straight vertical (or horizontal) lines, dividing the square into three congruent rectangles. The symmetry argument still applies here: one bifurcates on a single rectangle and then performs an odd extension to the whole square. However, this only accounts for two branches. For $u_1-u_2$, one sees the nodal lines are the two diagonal lines and thus the symmetry argument can still be applied (first bifurcate on one triangle and then extend to the square). The remaining branch corresponds to $u_1+u_2$, whose nodal line is represented in Figure 3.
	\begin{figure}[h]
		\centering
		\begin{subfigure}{.3\textwidth}
			\centering
			\includegraphics[width=.4\linewidth]{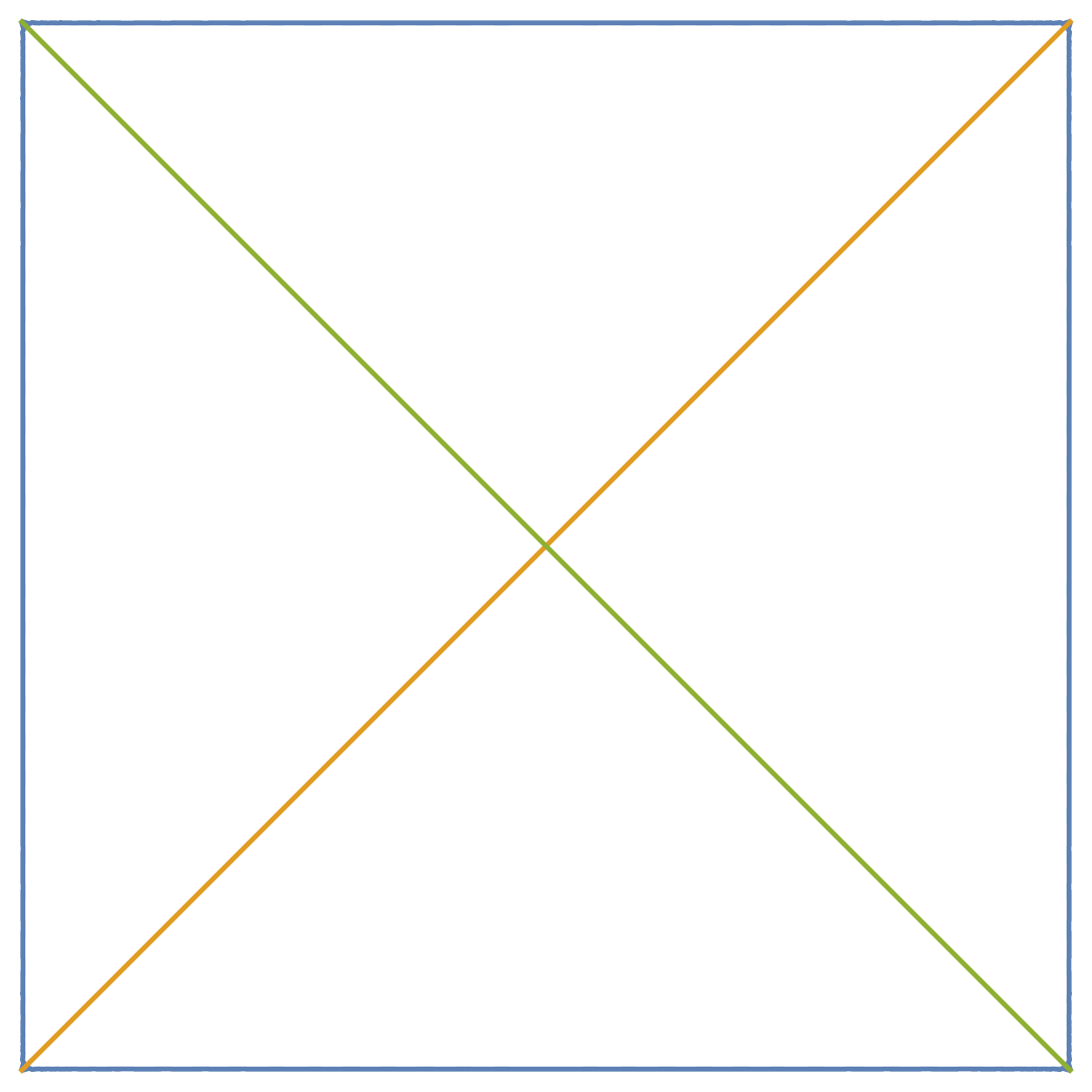}
			\caption{Nodal lines of $u_1-u_2$.}
		\end{subfigure}
		\begin{subfigure}{.3\textwidth}
			\centering
			\includegraphics[width=.4\linewidth]{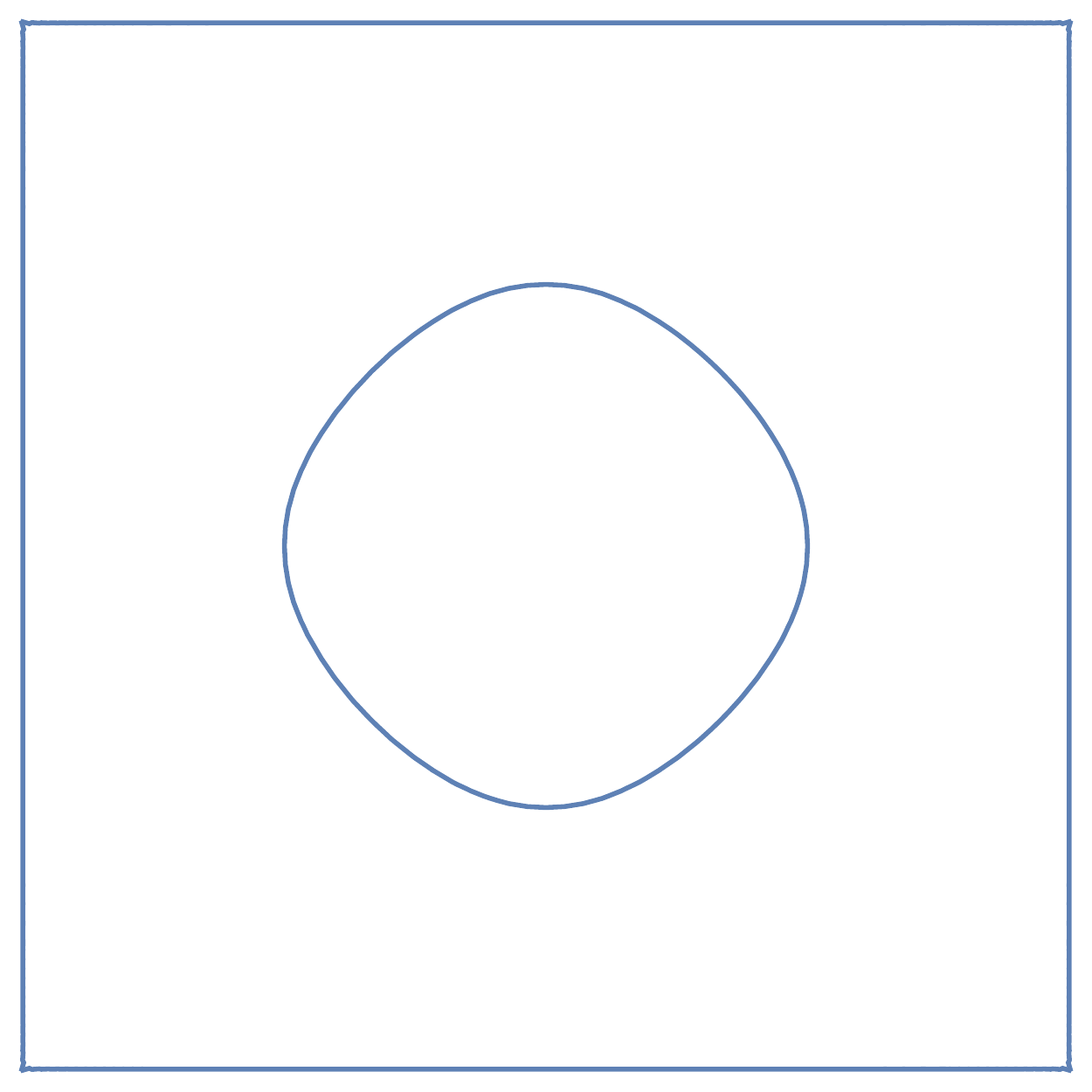}
			\caption{Nodal lines of $u_1+u_2$.}
		\end{subfigure}
		\caption{Bifurcations from the third eigenvalue of the square.}
	\end{figure}
	\vskip15pt
	\item (Another two-dimensional domain) In the rectangle $[0,1]\times[0,\sqrt{3/5}]$, the first double eigenvalue is the third one. The eigenspace is generated by
	$$
	u_1(x,y)=\sin(3\pi x)\sin(\sqrt{5}y/3),\quad u_2(x,y)=\sin(2\pi x)\sin(2\sqrt{5}y/3)
	$$
	The four bifurcation branches correspond to the four eigenvectors $u_1, u_2$ and $u_1\pm u_2$, whose contour plot can be found in Figure \ref{fig:3}.
	\begin{figure}[h]
		\centering
		\begin{subfigure}{.3\textwidth}
			\centering
			\includegraphics[width=.8\linewidth]{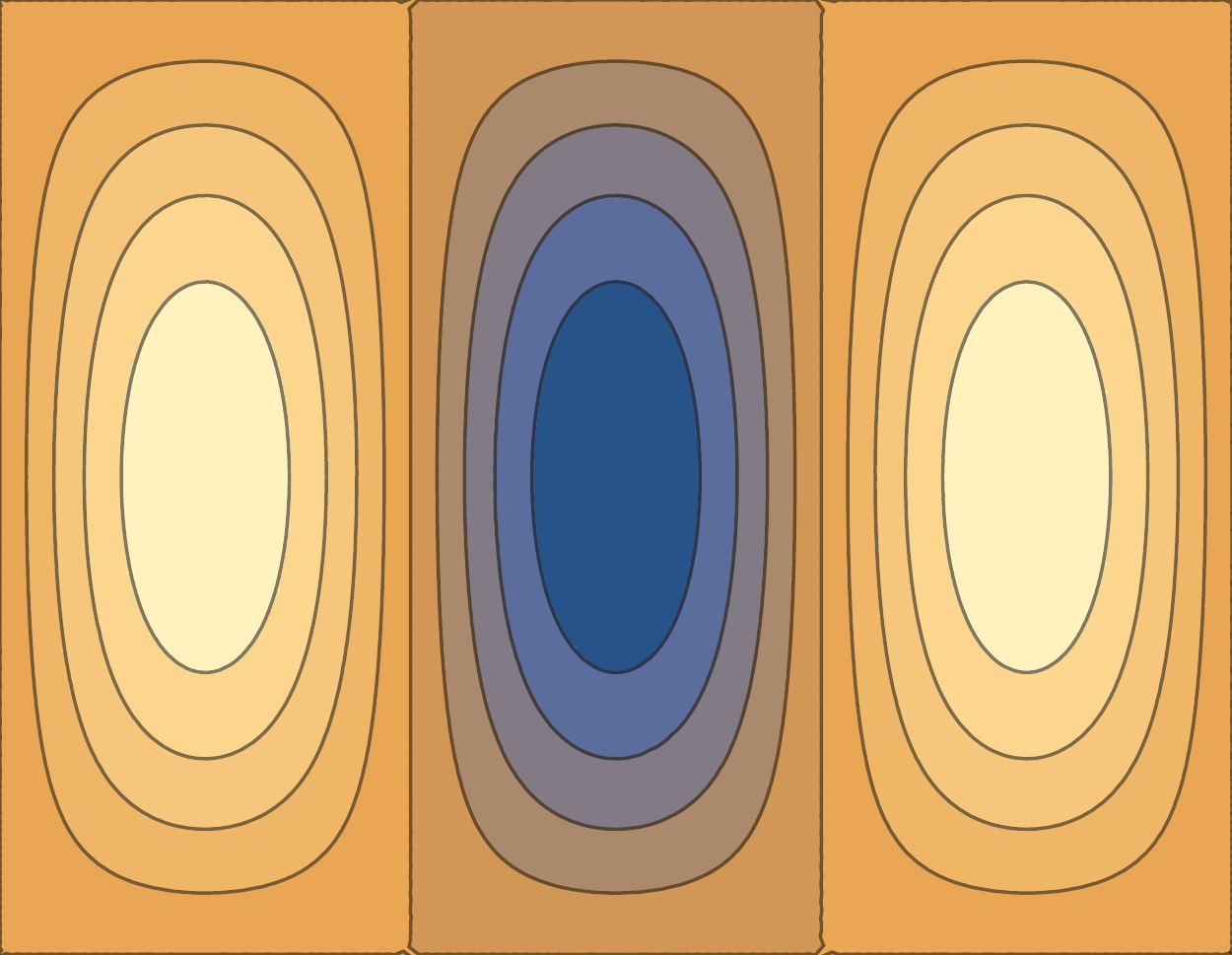}
			%			\caption{Two $\alpha$'s equal to 0\\ 3 branches}
			
		\end{subfigure}%
		\begin{subfigure}{.3\textwidth}
			\centering
			\includegraphics[width=.8\linewidth]{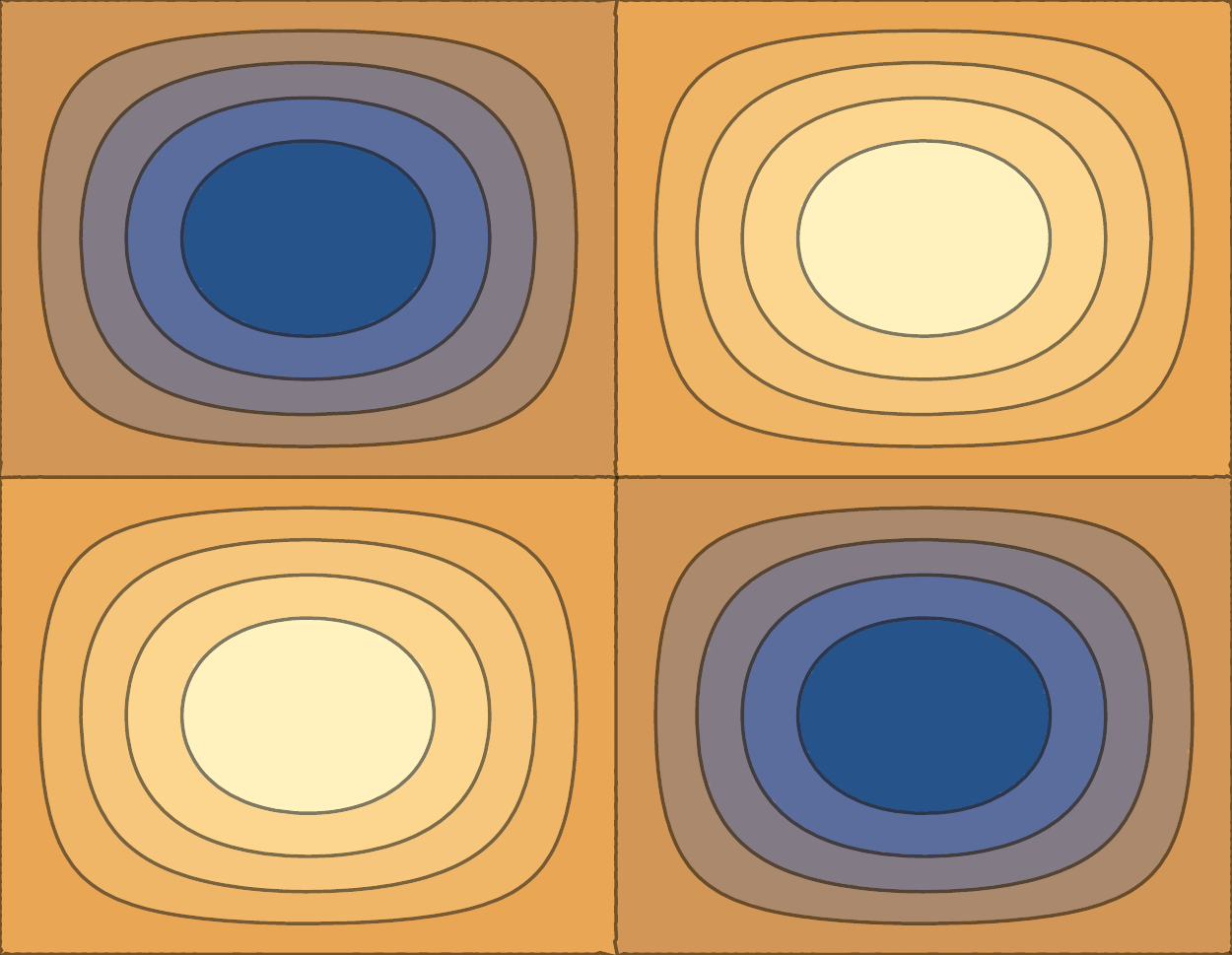}
			
		\end{subfigure}
		%	\end{figure}
	%			\begin{figure}[h]
		%		\centering
		\vskip15pt
		\begin{subfigure}{.3\textwidth}
			\centering
			\includegraphics[width=.8\linewidth]{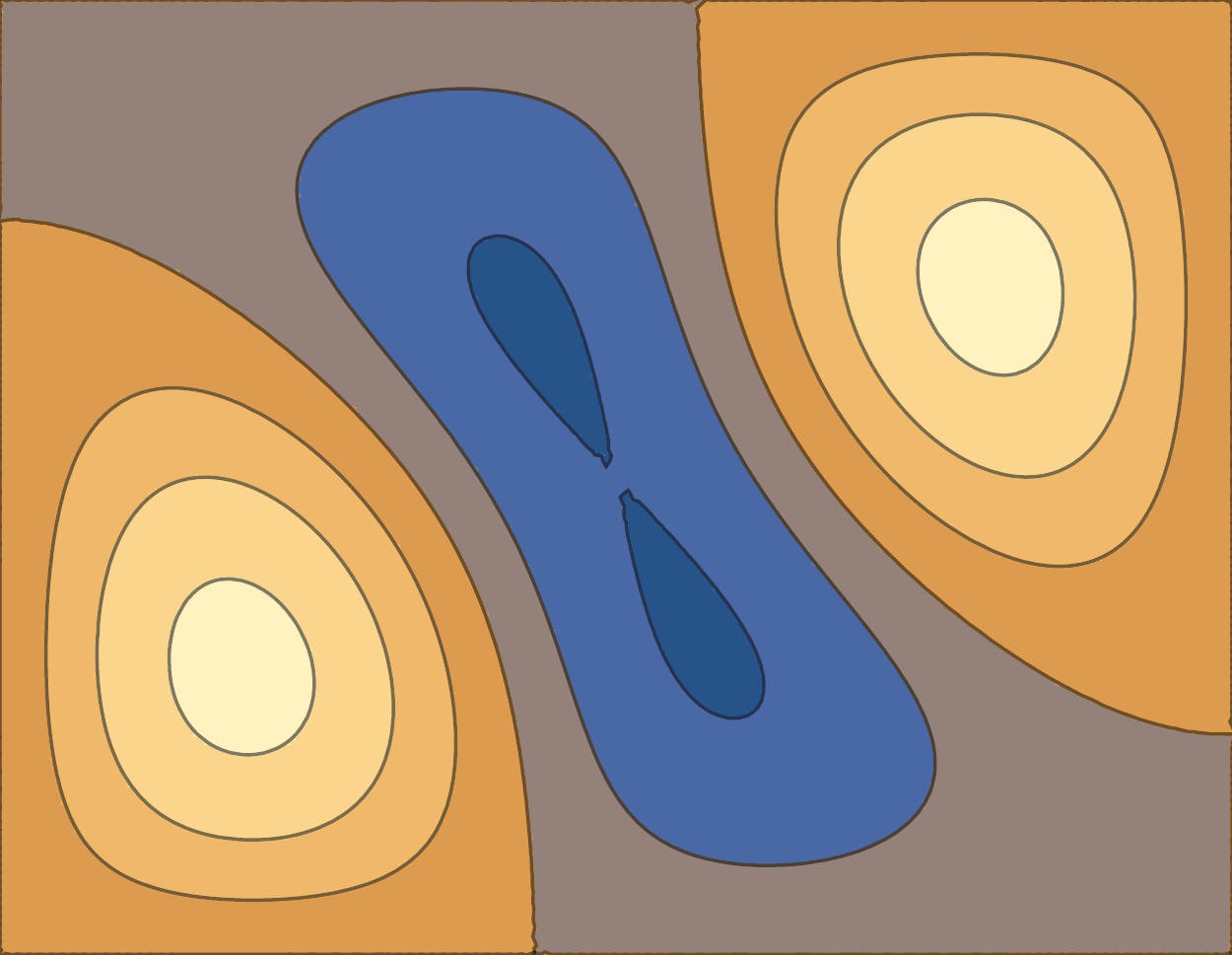}
			%			\caption{Two $\alpha$'s equal to 0\\ 3 branches}
			
		\end{subfigure}%
		\begin{subfigure}{.3\textwidth}
			\centering
			\includegraphics[width=.8\linewidth]{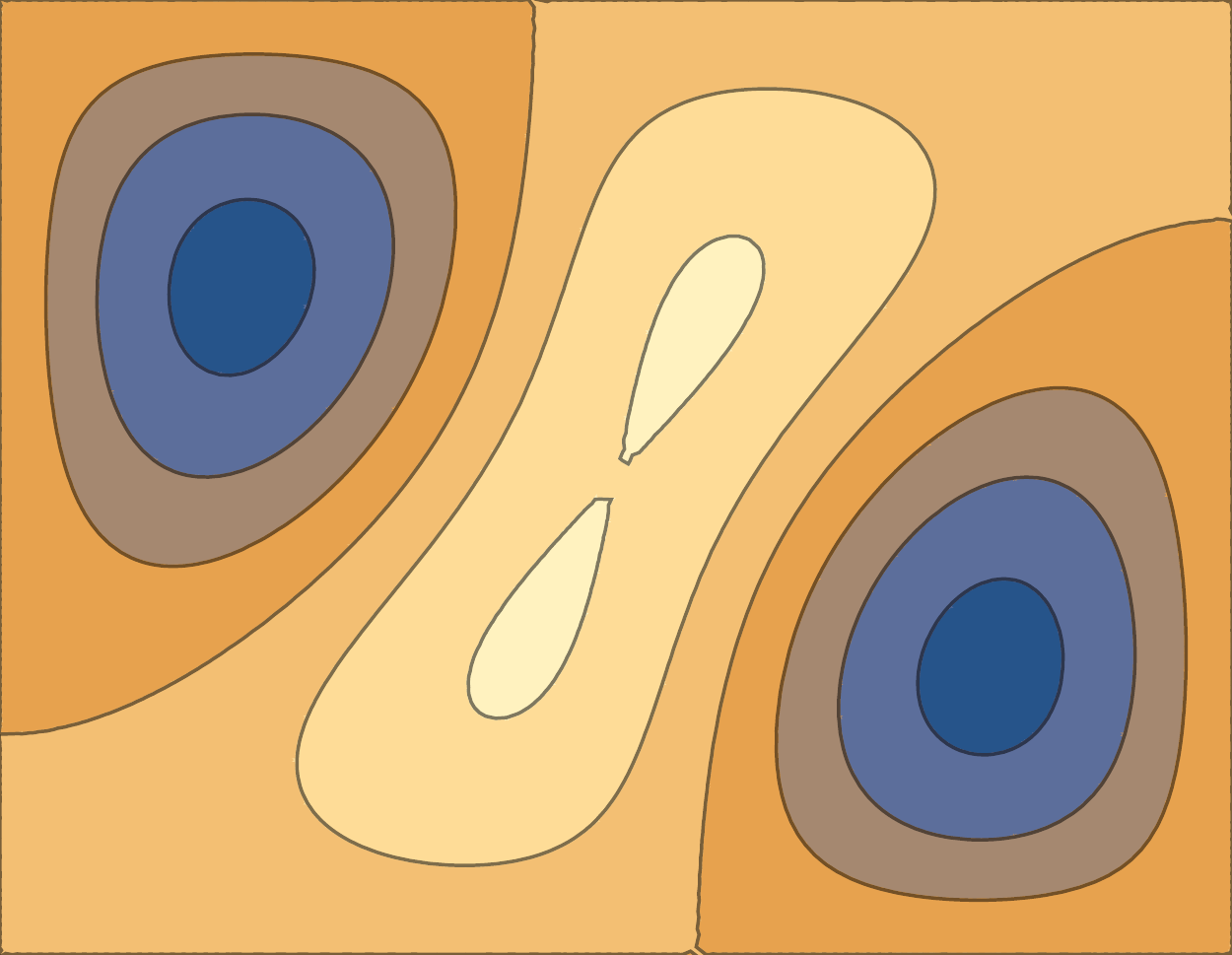}
			
		\end{subfigure}
		\caption{Contour plot of $u_1, u_2$ (top) and $u_1\pm u_2$ (bottom). The symmetry argument can only be applied to the first two cases.} 
		\label{fig:3}
	\end{figure}
\end{enumerate}
\vskip15pt\noindent \textbf{Acknowledgements.} The authors are indebted to Hugo Tavares, for many fruitful discussions in this topic, and also to the anonymous referee, whose comments led to a substantial improvement of the manuscript. S.C was partially supported by Fundação para a Ciência e Tecnologia, through the grant UID/MAT/04459/2019 and the project NoDES (PTDC/MAT-PUR/1788/2020). M.F. was partially supported by Fundação para a Ciência e Tecnologia, through the grant UIDB/04561/2020.
\bibliography{Biblioteca.bib}
\bibliographystyle{plain}

 \bigskip
 \bigskip
 
 \normalsize
 
 \begin{center}
	{\scshape Sim\~ao Correia}\\
{\footnotesize
	Center for Mathematical Analysis, Geometry and Dynamical Systems,\\
	Department of Mathematics,\\
	Instituto Superior T\'ecnico, Universidade de Lisboa\\
	Av. Rovisco Pais, 1049-001 Lisboa, Portugal\\
	simao.f.correia@tecnico.ulisboa.pt
}
  \bigskip
 
 	 	{\scshape M\'ario Figueira}\\
 	{\footnotesize
 		CMAF-CIO, Universidade de Lisboa\\
Edif\'{\i}cio C6, Campo Grande\\
1749-016 Lisboa, Portugal\\
{msfigueira@fc.ul.pt}
}

 \end{center} 
 
 \end{document}